\documentclass[12pt]{article}
\usepackage{amssymb}
\usepackage{amsmath}
\usepackage{amsthm}
\usepackage{authblk}
\usepackage{color}
\usepackage{comment}
\usepackage{geometry}		
\usepackage{graphicx}

\makeatletter
	\@addtoreset{equation}{section}
\makeatother

\let\originalleft\left
\let\originalright\right
\renewcommand{\left}{\mathopen{}\mathclose\bgroup\originalleft}
\renewcommand{\right}{\aftergroup\egroup\originalright}

\begin{document}

\newcommand\cF{\mathcal{F}}
\newcommand{\co}{o}
\newcommand\cO{O}
\newcommand{\ee}{\varepsilon}

\newcommand{\myStep}[2]{{\bf Step #1} --- #2\\}

\newtheorem{theorem}{Theorem}[section]
\newtheorem{corollary}[theorem]{Corollary}
\newtheorem{lemma}[theorem]{Lemma}
\newtheorem{proposition}[theorem]{Proposition}

\theoremstyle{definition}
\newtheorem{definition}{Definition}[section]

\title{Extended normal forms for one-dimensional border-collision bifurcations}
\author[$\dagger$]{P.A.~Glendinning}
\author[$\ddagger$]{D.J.W.~Simpson}
\affil[$\dagger$]{Department of Mathematics, University of Manchester, Manchester, UK}
\affil[$\ddagger$]{School of Mathematical and Computational Sciences, Massey University, Palmerston North, New Zealand}

\maketitle

\begin{abstract}

The border-collision normal form describes the local dynamics in continuous systems with switches when a fixed point intersects a switching surface. For one-dimensional cases where the bifurcation creates or destroys only fixed points and period-two orbits, we show that the standard local equivalence of normal forms, topological conjugacy, can be replaced by differentiable conjugacy provided an extra term is added to the normal form. In these cases topological conjugacy is so weak that a range of values can be used for the coefficients in the normal form. The extension to differentiable conjugacy explains why the usual choice of slopes in the standard normal form is privileged. This highlights the importance of differentiable conjugacies and the need for extended normal forms.

\end{abstract}

\section{Introduction}
\label{sec:intro}

Bifurcation theory is underpinned by normal forms;
these are simple evolution rules with local equivalence to general classes of dynamical systems
where equivalence usually refers to a topological conjugacy.
However, for border-collision bifurcations, which occur in systems with thresholds, switches, or contact events,
the standard normal form is a piecewise-linear map that may fail to be topologically conjugate to the original system.
Here we consider border-collision bifurcations in one dimension
that mimic saddle-node or period-doubling bifurcations,
and show that equivalence via a stronger \emph{differentiable} conjugacy can be achieved if an appropriate 
quadratic term is added to the normal form. We call the result an extended normal form,
as in our earlier work \cite{GlSi23} for classical bifurcations of smooth maps.
If the border-collision bifurcation brings no significant change to the dynamics 
then the quadratic term is not needed.
In either case, for a given map the extended normal form can be constructed by choosing parameters
so that the two maps have periodic solutions
with the same stability multipliers
and the same ratio of lower and upper derivatives at their points of non-differentiability.
The latter condition ensures the domain of a local differentiable conjugacy can be
extended through the non-differentiable point.
This work aims to initiate
a rigorous normal form theory for border-collision bifurcations.

The usual normal forms for saddle-node bifurcations
and one-dimensional border-collision bifurcations are
\begin{align}
y \mapsto \nu + y - y^2,
\label{eq:sn}
\end{align}
and
\begin{align}
y \mapsto \begin{cases} \nu + s_L y, & y \le 0, \\ \nu + s_R y, & y \ge 0, \end{cases}
\label{eq:pwl}
\end{align}
respectively.
Border-collision bifurcations allow many different dynamical transitions \cite{DiBu08,ItTa79}. 
If $0 < s_R < 1 < s_L$ 
the bifurcation mimics a saddle-node bifurcation, see Fig.~\ref{fig:pwsDiffInit}.

\begin{figure}[h!]
\begin{center}
\includegraphics[height=5.25cm]{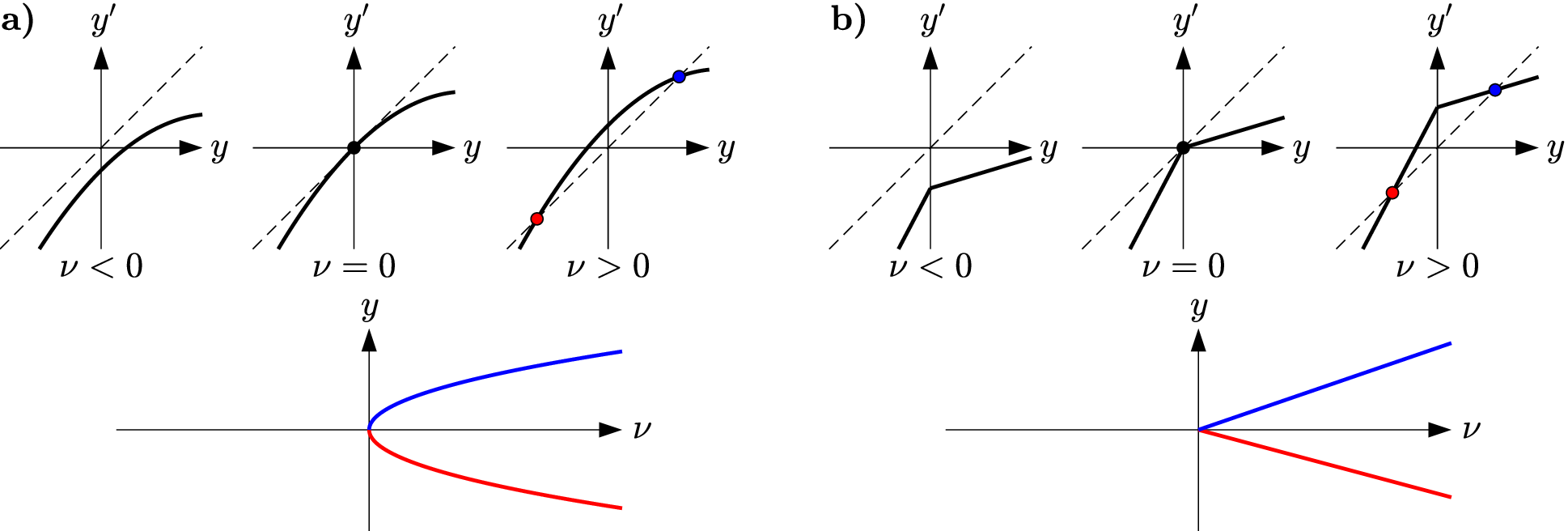}
\caption{
Bifurcation diagrams of the quadratic map \eqref{eq:sn} in panel (a)
and the piecewise-linear map \eqref{eq:pwl} in panel (b) for values of $\nu$ near $0$.
Stable [unstable] fixed points are coloured blue [red].
Each diagram is accompanied by three plots showing the map
before, at, and after the bifurcation.
\label{fig:pwsDiffInit}
} 
\end{center}
\end{figure}

The quadratic map \eqref{eq:sn} applies to smooth maps
where a saddle-node bifurcation occurs when a fixed point attains a stability multiplier of $1$
and certain genericity conditions are met \cite{GuHo83,Ku04}.
The piecewise-linear map \eqref{eq:pwl} applies to one-dimensional piecewise-smooth continuous maps
where a border-collision bifurcation occurs when a fixed point
reaches a point where the map non-differentiable \cite{NuYo95,SuAv16}.
Such bifurcations occur in mathematical models of diverse systems
including power electronics \cite{BaKa00} and economics \cite{GaSu08}. 
The maps \eqref{eq:sn} and \eqref{eq:pwl} with $0 < s_R < 1 < s_L$
are indeed normal forms in the sense of being locally topologically conjugate
to the general classes of maps under consideration.

But topological conjugacy is weak.
The map \eqref{eq:sn} is topologically conjugate to \eqref{eq:pwl} for any
values of $s_L$ and $s_R$ with $0 < s_R < 1 < s_L$,
and the parameters $\nu$ may also differ provided their sign is the same (see \S\ref{sec:topConj} for details).
Thus the precise values of $s_L$ and $s_R$ do not matter
even though in applications
they are usually chosen to match those of the original system on the switching surface.   
Normal forms are more informative when the conjugacy is a diffeomorphism
because in this case it preserves the stability multipliers of periodic solutions, for instance.

The quadratic map \eqref{eq:sn} usually does not provide a differentiable conjugacy.
However, if we add a certain parameter-dependent cubic term to
\eqref{eq:sn} the resulting map does provide a differentiable conjugacy \cite{GlSi23}.
The purpose of this paper is to obtain similar results for border-collision bifurcations of one-dimensional maps
\begin{equation}
f(x;\mu) = \begin{cases} f_L(x;\mu), & x \le 0, \\ f_R(x;\mu), & x \ge 0, \end{cases}
\label{eq:f}
\end{equation}
where $f_L$ are $f_R$ are smooth.
There are several cases to consider.
If a border-collision bifurcation of \eqref{eq:f} causes a fixed point to cross
from one side of the non-differentiable point $x=0$ to the other
and no other invariant sets are involved, then \eqref{eq:pwl} can be used to obtain a differentiable conjugacy.
If instead the bifurcation mimics a saddle-node bifurcation or a period-doubling bifurcation,
then \eqref{eq:pwl} is insufficient.
For this reason we introduce the {\em extended normal form}
\begin{equation}
g(y;\nu,s_L,s_R,t) = \begin{cases} \nu + s_L y + t y^2, & y \le 0, \\ \nu + s_R y, & y \ge 0, \end{cases}
\label{eq:pwq}
\end{equation}
where the $y^2$-term provides us with enough degrees of freedom
to accommodate the latter cases.
Our results do not cover scenarios with
$s_L > 0$ and $s_R < -1$ (or $s_L < -1$ and $s_R > 0$)
for which the border-collision bifurcation can create higher period solutions and chaotic sets.

Our proofs use a similar methodology to \cite{GlSi23} for the smooth setting.
They are achieved by first showing that the parameters of the extended normal form can be chosen
so that its invariant sets have the same stability multipliers as the corresponding invariant sets of the general map.
This equality needs to hold on intervals of parameter values, and we use the implicit function theorem to achieve this.
The novelty of the piecewise-smooth setting is that the two maps must also have the same ratio
of lower and upper derivatives at $x=0$ and $y=0$:
\begin{equation}
\frac{s_L}{s_R} = \frac{\frac{\partial f_L}{\partial x}(0,\mu)}{\frac{\partial f_R}{\partial x}(0,\mu)}.
\label{eq:slopeRatiosMatch}
\end{equation}
To construct differentiable conjugacies on intervals containing $x=0$,
we begin with a local conjugacy that exists by a linearisation theorem of Sternberg \cite{St57b},
then extend its domain outwards using images and preimages of the maps as done by Belitskii \cite{Be86}.
The key observation is that \eqref{eq:slopeRatiosMatch} ensures the conjugacy (and its inverse)
remains differentiable as it is extended through $x=0$.
The necessity of this condition has also been reported for circle maps \cite{KhKh03,KhTe13}.

The remainder of the paper is organised as follows.
We begin in \S\ref{sec:mainResults} by formally stating our main results as three theorems.
Then in \S\ref{sec:topConj} we consider topological conjugacies
and provide a precise result that explains why \eqref{eq:sn} and \eqref{eq:pwl} with $0 < s_R < 1 < s_L$
are topologically conjugate to each other. 
In \S\ref{sec:diffConj} we construct differentiable conjugacies between smooth maps,
then in \S\ref{sec:construction} use these results to construct differentiable conjugacies between piecewise-smooth maps.
This is a large task: ten separate constructions are required to accommodate all scenarios needed to prove the theorems,
which is then done in \S\ref{sec:matching}.
In \S\ref{sec:nd} we speculate on extensions to higher dimensional maps,
and in \S\ref{sec:conc} provide a final discussion.
Throughout the paper $\cO$ and $\co$ represent big-O and little-o notation:
a function of $u \in \mathbb{R}$ is $\cO \left( u^k \right)$ if it only contains terms that are order $k$ or higher,
and is $\co \left( u^k \right)$ if it only contains terms that are higher than order $k$.

\section{Main results}
\label{sec:mainResults}

We first clarify the notion of topological and differentiable conjugacies for one-dimensional maps.

\begin{definition}
Let $\Omega, \Psi \subset \mathbb{R}$ be open intervals
and let $f : \Omega \to \mathbb{R}$ and $g : \Psi \to \mathbb{R}$ be continuous maps.
Then $f$ on $U \subset \Omega$ and $g$ on $V \subset \Psi$ are {\em topologically conjugate}
if there exists a homeomorphism $h : U \to V$ such that
\begin{equation}
h(f(x)) = g(h(x)),
\label{eq:conjugacyRelation}
\end{equation}
for all $x \in U$ for which $f(x) \in U$.
If $h$ is also a diffeomorphism then $f$ on $U$ and $g$ on $V$ are {\em differentiably conjugate}.
\label{df:conj}
\end{definition}

Conjugacies convert orbits of one map to orbits of the other map.
For instance in Definition \ref{df:conj} if $x^* \in U$ is a fixed point of $f$,
then \eqref{eq:conjugacyRelation} implies $y^* = h(x^*) \in V$ is a fixed point of $g$.
Furthermore, if $f$ is differentiable at $x^*$, $g$ is differentiable at $y^*$, and $h$ is a diffeomorphism,
then \eqref{eq:conjugacyRelation} implies $\frac{d f}{d x}(x^*) = \frac{d g}{d y}(y^*)$,
i.e.~the fixed points have the same stability multiplier.

Now consider a piecewise-smooth map $f$ of the form \eqref{eq:f}.
We suppose $f$ is defined for all $(x;\mu) \in \Omega \times I$,
where $\Omega$ and $I$ are open intervals containing $0$,
and that $f_L$ and $f_R$
can be smoothly extended to a neighbourhood of $x=0$ (so the derivatives of $f_L$ and $f_R$ are well-defined at $x=0$).
We assume $f$ is continuous at $x=0$, that is
\begin{equation}
f_L(0;\mu) = f_R(0;\mu), \qquad \text{for all $\mu \in I$}.
\label{eq:continuity}
\end{equation}
In order for $\mu = 0$ to correspond to a border-collision bifurcation,
we suppose $x=0$ is a fixed point of $f$ when $\mu = 0$, i.e.
\begin{equation}
f_L(0;0) = f_R(0;0) = 0.
\label{eq:bcbAssumption}
\end{equation}
Now let
\begin{align}
a_L &= \frac{\partial f_L}{\partial x}(0;0), &
a_R &= \frac{\partial f_R}{\partial x}(0;0),
\label{eq:aLaR}
\end{align}
be the left and right slopes of the map at the bifurcation point.
In order for $\mu$ to unfold the bifurcation in a generic fashion, we require
\begin{equation}
\beta = \frac{\partial f_L}{\partial \mu}(0;0) = \frac{\partial f_R}{\partial \mu}(0;0)
\label{eq:beta}
\end{equation}
to be non-zero.
The results below assume $\beta > 0$ to fix the direction of the bifurcation.

The bifurcation behaves very differently for different values of $a_L$ and $a_R$.
The three theorems below correspond to the coloured regions in Fig.~\ref{fig:pwsDiffBifSet} numbered $1$, $2$, and $3$.
The same results apply to the unnumbered coloured regions
but with stability reversed and/or a different direction for the bifurcation.
The white regions include parameter values where a chaotic attractor
is created at the border-collision bifurcation
and are not covered by the theorems.
Proofs of the theorems are provided in \S\ref{sec:matching}.

The first theorem applies to the six regions numbered $1$.
For these the border-collision bifurcation is trivial
in the sense that as the value of $\mu$ is varied through $0$,
a fixed point crosses $x=0$ and no other invariant sets are involved.
The theorem shows that a differentiable conjugacy to the extended normal form \eqref{eq:pwq} can be achieved using $t = 0$
with which \eqref{eq:pwq} reverts to the piecewise-linear form \eqref{eq:pwl}.
See Fig.~\ref{fig:pwsDiffConjBifDiag_a} for a typical bifurcation diagram.
Here and throughout the paper we label fixed points and other important values on the $45^\circ$-degree line (dashed).

\begin{figure}[h!]
\begin{center}
\includegraphics[height=8cm]{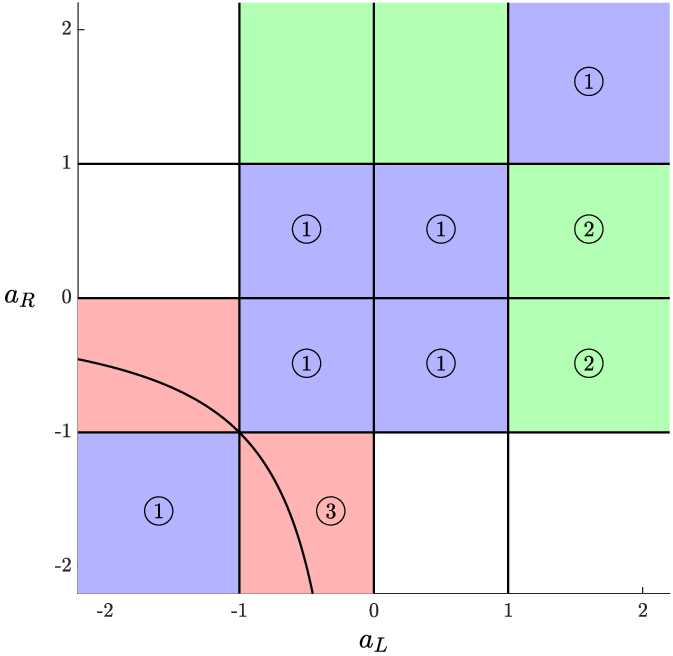}
\caption{
The space of parameter points $(a_L,a_R)$ for border-collision bifurcations of a map \eqref{eq:f},
where $a_L$ and $a_R$ are the left and right slopes \eqref{eq:aLaR}.
Theorem \ref{th:noBif} applies to the regions numbered 1,
Theorem \ref{th:snLike} applies to the regions numbered 2,
and Theorem \ref{th:pdLike} applies to the region numbered 3.
The unnumbered coloured regions can be accommodated by symmetry.
The curved boundary is where $a_L a_R = 1$ in the third quadrant.
\label{fig:pwsDiffBifSet}
} 
\end{center}
\end{figure}

\begin{figure}[h!]
\begin{center}
\includegraphics[height=7cm]{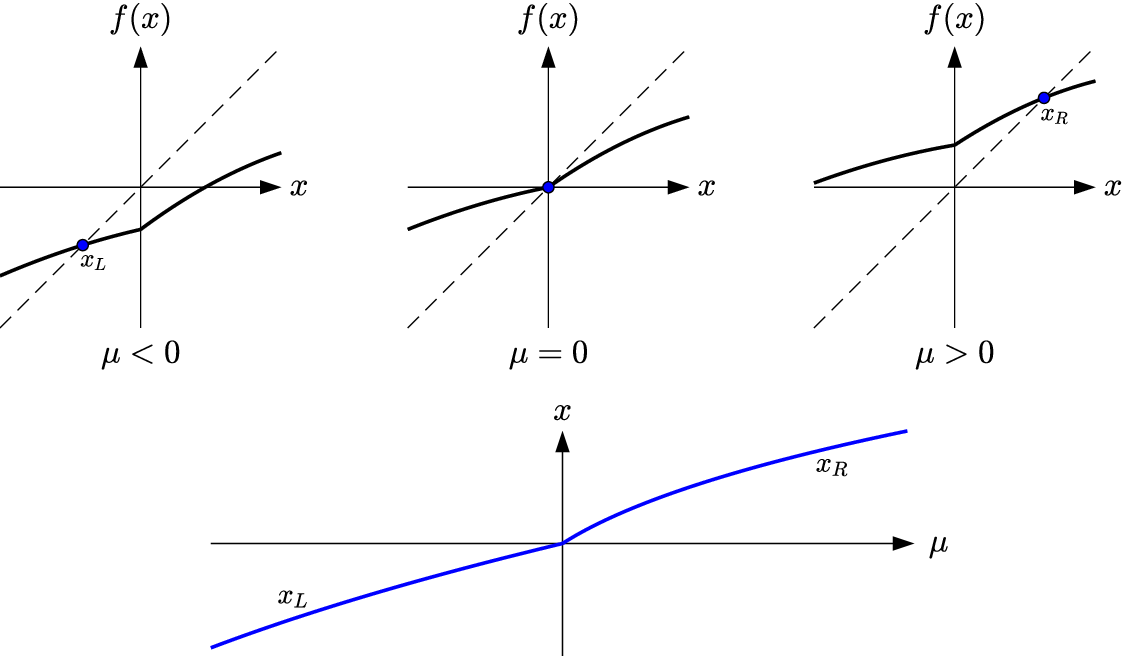}
\caption{
A bifurcation diagram and representative cobweb diagrams
of a piecewise-smooth map \eqref{eq:f} satisfying the conditions of Theorem \ref{th:noBif}
with $0 < a_L < 1$ and $0 < a_R < 1$.
For small $\mu < 0$ the map has a stable fixed point $x_L < 0$,
while for small $\mu > 0$ it has a stable fixed point $x_R > 0$.
\label{fig:pwsDiffConjBifDiag_a}
} 
\end{center}
\end{figure}

\begin{theorem}[Trivial border-collision bifurcations]
Suppose $f$ of the form \eqref{eq:f} is piecewise-$C^2$ and satisfies \eqref{eq:continuity}, \eqref{eq:bcbAssumption},
$\beta > 0$, $(a_L-1)(a_R-1) > 0$, $(a_L+1)(a_R+1) > 0$, $a_L \ne 0$, and $a_R \ne 0$.
Let $\nu(\mu) = f(0;\mu)$.
Then there exist $\mu_0 > 0$,
continuous functions $s_L(\mu) = a_L + \cO(\mu)$ and $s_R(\mu) = a_R + \cO(\mu)$,
and neighbourhoods $N$ and $M(\mu)$ of $0$,
such that $f(x;\mu)$ on $N$ and $\tilde{g}(y;\mu) = g(y;\nu(\mu),s_L(\mu),s_R(\mu),0)$ on $M(\mu)$ are differentiably conjugate
for all $\mu \in (-\mu_0,\mu_0)$.
Moreover, for any $\delta > 0$
we can choose $\mu_0 > 0$, $N = (-p,p)$, and $M(\mu) = (q^-(\mu),q^+(\mu))$ so that
\begin{equation}
(1-\delta) p < \left| q^\pm(\mu) \right| < (1+\delta) p,
\label{eq:theBound}
\end{equation}
for all $\mu \in (-\mu_0,\mu_0)$.
\label{th:noBif}
\end{theorem}

Notice Theorem \ref{th:noBif} (also Theorems \ref{th:snLike} and \ref{th:pdLike}) fixes $\nu(\mu) = f(0;\mu)$.
We could have instead taken $\nu(\mu)$ to be any increasing function of $\mu$ with $\nu(0) = 0$,
and the conclusions of the theorem would still hold.
This is because $f$ is asymptotically piecewise-linear so
the magnitude of $\mu$ primarily affects the spatial scale of the dynamics, not its qualitative properties.
We chose $\nu(\mu) = f(0;\mu)$ for simplicity, concreteness,
and so that the dynamics of $f$ and $g$ occurs on the roughly the same scale for all $\mu \in (-\mu_0,\mu_0)$.
Indeed this is formalised by the bound \eqref{eq:theBound}.
Here the neighbourhood $N$, for the map $f$, is taken for simplicity to be the symmetric interval $(-p,p)$.
The corresponding neighbourhood $M$, for the map $g$, varies with $\mu$,
but \eqref{eq:theBound} shows that the amount by which it varies from $N$ (in ratio)
can be made arbitrarily small by considering sufficiently small values of $\mu$.
This reflects the fact that $\tilde{g}$ has been designed to match $f$ to first-order,
hence the conjugacy function $h(x;\mu)$ can be made near-identity,
so $M(\mu) = h(N;\mu)$ is a small perturbation of $N$.
This result is also present in the next two theorems.

Now we address the two regions in Fig.~\ref{fig:pwsDiffBifSet} numbered 2
(the other green regions can be accommodated via the substitution $x \mapsto -x$).
Fig.~\ref{fig:pwsDiffConjBifDiag_b} shows a typical bifurcation diagram.
Locally $f$ and $g$ each have no fixed points for $\mu < 0$ and two fixed points for $\mu > 0$.
We denote these fixed points $x_L < 0$ and $x_R > 0$ for the map $f$, and $y_L < 0$ and $y_R > 0$ for the map $g$.

\begin{figure}[h!]
\begin{center}
\includegraphics[height=7cm]{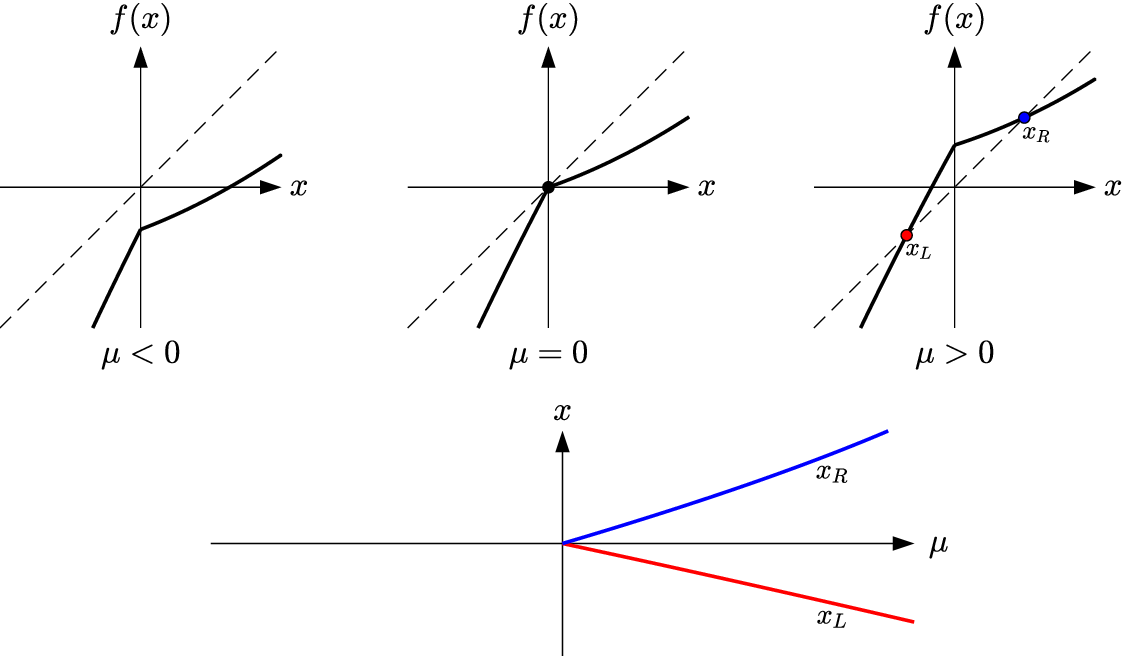}
\caption{
A bifurcation diagram and representative cobweb diagrams
of a piecewise-smooth map \eqref{eq:f} satisfying the conditions of Theorem \ref{th:snLike}
in the invertible case $a_R > 0$.
For small $\mu > 0$ the map has a stable fixed point $x_R > 0$
and an unstable fixed point $x_L < 0$.
\label{fig:pwsDiffConjBifDiag_b}
} 
\end{center}
\end{figure}

\begin{theorem}[Saddle-node-like border-collision bifurcations]
Suppose $f$ of the form \eqref{eq:f} is piecewise-$C^3$ and satisfies \eqref{eq:continuity}, \eqref{eq:bcbAssumption},
$\beta > 0$, $a_L > 1$, and $0 < |a_R| < 1$.
Let $\nu(\mu) = f(0;\mu)$.
Then there exist $\mu_0 > 0$,
continuous functions $s_L(\mu) = a_L + \cO(\mu)$, $s_R(\mu) = a_R + \cO(\mu)$, and $t(\mu)$,
and neighbourhoods $N$ and $M(\mu)$ of $0$,
such for $\tilde{g}(y;\mu) = g(y;\nu(\mu),s_L(\mu),s_R(\mu),t(\mu))$
\begin{enumerate}
\item
if $\mu \in (-\mu_0,0]$ then $f$ on $N$ and $\tilde{g}$ on $M(\mu)$ are differentiably conjugate,
\item
if $\mu \in (0,\mu_0)$ and $a_R > 0$
then $f$ on $N$ and $\tilde{g}$ on $M(\mu)$ are differentiably conjugate
on the basins of attraction of their stable fixed points
and on the basins of repulsion of the their unstable fixed points,
\item
if $\mu \in (0,\mu_0)$ and $a_R < 0$
then $f$ on $\left( x_L, f_R^{-1}(x_L) \right)$
and $\tilde{g}$ on $\left( y_L, g_R^{-1}(y_L) \right)$ are differentiably conjugate,
and $f$ on $\left\{ x \in N \,|\, x < 0 \right\}$
and $\tilde{g}$ on $\left\{ y \in M(\mu) \,|\, y < 0 \right\}$ are differentiably conjugate.
\end{enumerate}
Moreover, for any $\delta > 0$
we can choose $\mu_0 > 0$, $N = (-p,p)$, and $M(\mu) = (q^-(\mu),q^+(\mu))$ so that
\eqref{eq:theBound} holds for all $\mu \in (-\mu_0,\mu_0)$.
\label{th:snLike}
\end{theorem}

For small $\mu > 0$, Theorem \ref{th:snLike} gives differentiable conjugacies on subsets of $N$,
instead of the entirety of this neighbourhood,
because the domain of a differentiable conjugacy cannot in general be extended through a fixed point.
To make plain what these subsets are, suppose $N = (-p,p)$,
where $p > 0$ (we can always take $N$ to be symmetric about $0$).
We always assume $\mu > 0$ is small enough that $x_L, x_R \in (-p,p)$.
If $0 < a_R < 1$, as in Fig.~\ref{fig:pwsDiffConjBifDiag_b}, then $f$ is locally invertible and
the subsets of $N$ described in (ii)
are the intervals $(x_L,p)$ (the basin of attraction of $x_R$)
and $(-p,x_R)$ (the basin of repulsion of $x_L$).
If $-1 < a_R < 0$, as in Fig.~\ref{fig:pwsDiffConjAlt}, then $f$ is non-invertible and
the subsets of $N$ described in (iii)
are the intervals $\left( x_L, f_R^{-1}(x_L) \right)$ and $(-p,0)$.

\begin{figure}[h!]
\begin{center}
\includegraphics[height=4.375cm]{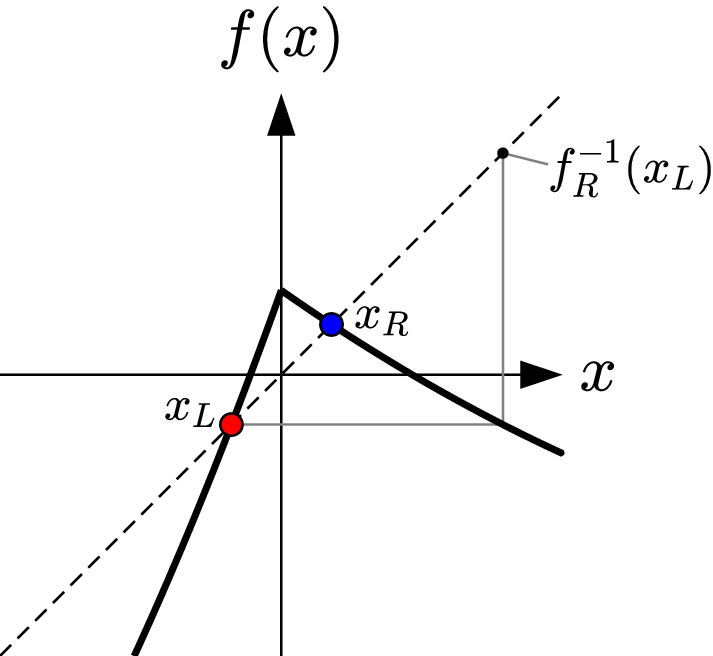} 
\caption{
A cobweb diagram of a piecewise-smooth map \eqref{eq:f} satisfying the conditions of Theorem \ref{th:snLike}
with small $\mu > 0$ in the non-invertible case $a_R < 0$.
\label{fig:pwsDiffConjAlt}
} 
\end{center}
\end{figure}

In the red regions of Fig.~\ref{fig:pwsDiffBifSet} the map $f$ is locally invertible
and the border-collision bifurcation generates a period-two solution.
Of these, Theorem \ref{th:pdLike} applies to the region numbered $3$
(the other regions can be accommodated via $x \mapsto -x$ and considering $f^{-1}$ instead of $f$).
Fig.~\ref{fig:pwsDiffConjBifDiag_c} shows a typical bifurcation diagram.

\begin{figure}[h!]
\begin{center}
\includegraphics[height=7cm]{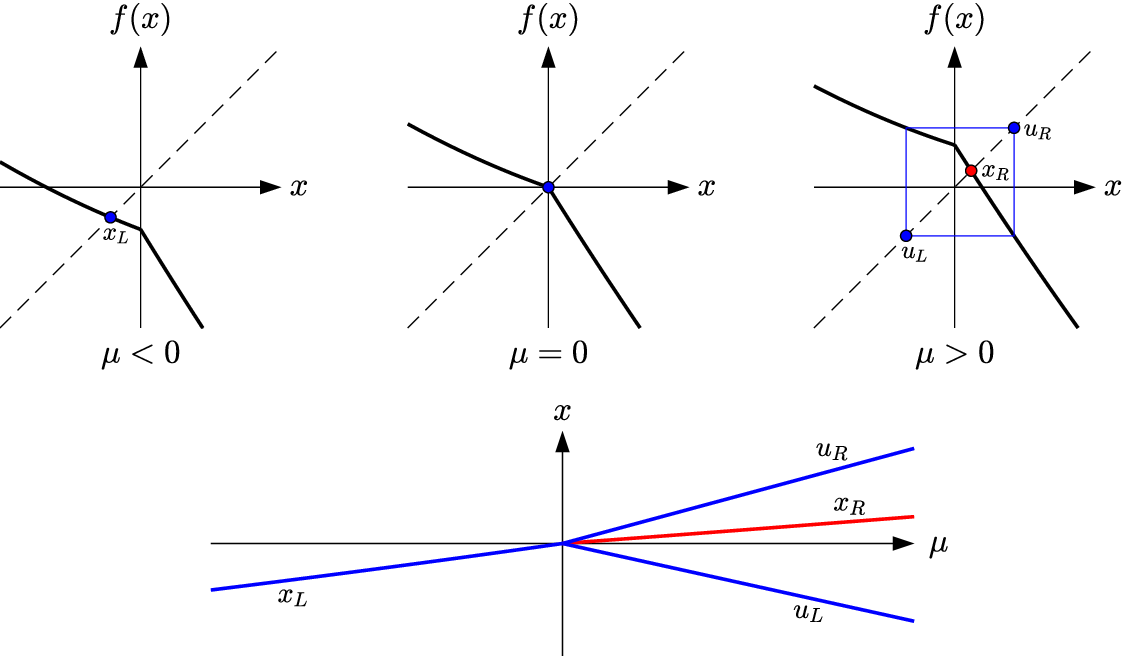}
\caption{
A bifurcation diagram and representative cobweb diagrams
of a piecewise-smooth map \eqref{eq:f} satisfying the conditions of Theorem \ref{th:pdLike}.
For small $\mu < 0$ the map has a stable fixed point $x_L < 0$,
while for small $\mu > 0$ it has an unstable fixed point $x_R > 0$
and a stable period-$2$ solution $\{ u_L, u_R \}$.
\label{fig:pwsDiffConjBifDiag_c}
} 
\end{center}
\end{figure}

\begin{theorem}[Period-doubling-like border-collision bifurcations]
Suppose $f$ of the form \eqref{eq:f} is piecewise-$C^3$ and satisfies \eqref{eq:continuity}, \eqref{eq:bcbAssumption},
$\beta > 0$, $a_R < -1$, and $\frac{1}{a_R} < a_L < 0$.
Let $\nu(\mu) = f(0;\mu)$.
Then there exist $\mu_0 > 0$,
continuous functions $s_L(\mu) = a_L + \cO(\mu)$, $s_R(\mu) = a_R + \cO(\mu)$, and $t(\mu)$,
and neighbourhoods $N$ and $M(\mu)$ of $0$,
such for $\tilde{g}(y;\mu) = g(y;\nu(\mu),s_L(\mu),s_R(\mu),t(\mu))$
\begin{enumerate}
\item
if $\mu \in (-\mu_0,0]$ then $f$ on $N$ and $\tilde{g}$ on $M(\mu)$ are differentiably conjugate,
\item
if $\mu \in (0,\mu_0)$ then
$f$ on $N$ and $\tilde{g}$ on $M(\mu)$ are differentiably conjugate on the basins of repulsion of their fixed points,
and on the basins of attraction of their period-two solutions.
\end{enumerate}
Moreover, for any $\delta > 0$
we can choose $\mu_0 > 0$, $N = (-p,p)$, and $M(\mu) = (q^-(\mu),q^+(\mu))$ so that
\eqref{eq:theBound} holds for all $\mu \in (-\mu_0,\mu_0)$.
\label{th:pdLike}
\end{theorem}

\section{Topological conjugacies}
\label{sec:topConj}

Before we delve into constructions for differentiable conjugacies,
it is worth explaining why topological conjugacies provide such a weak criterion to impose upon a normal form.
The following result illustrates this weakness for saddle-node-like border-collision bifurcations.
It is analogous to Theorem \ref{th:snLike},
but instead of the extended normal form \eqref{eq:pwq},
uses the piecewise-linear map \eqref{eq:pwl} and the smooth map \eqref{eq:sn}.

\begin{proposition}
Suppose $f$ of the form \eqref{eq:f} is piecewise-$C^1$ and satisfies \eqref{eq:continuity}, \eqref{eq:bcbAssumption},
$\beta > 0$, $a_L > 1$, and $0 < a_R < 1$.
Let $s_L > 1$, $0 < s_R < 1$, and $\nu(\mu) = f(0;\mu)$.
Then there exist $\mu_0 > 0$ 
and neighbourhoods $N$, $M(\mu)$, and $P(\mu)$ of $0$,
such that $f$ on $N$,
\eqref{eq:pwl} on $M(\mu)$,
and \eqref{eq:sn} on $P(\mu)$ are topologically conjugate for all $\mu \in (-\mu_0,\mu_0)$.
\label{pr:topConj}
\end{proposition}

We omit a proof of Proposition \ref{pr:topConj} as it is a simple consequence of a classical result
that two increasing maps $f$ and $g$ on $\mathbb{R}$ are topologically conjugate if they have the same number of fixed points
and the signs of $f(x)-x$ and $g(y)-y$ are the same between corresponding fixed points, see Theorem 2.2 of \cite{GlSi23}.

There are many variations on Proposition \ref{pr:topConj} that could be stated.
For example, the value of $\nu$ in the piecewise-linear map \eqref{eq:pwl}
can to be converted to $-1$, $0$, or $1$ by a linear rescaling of $y$,
so for \eqref{eq:pwl} we in fact only need $\nu(\mu)$ to have the same sign as $\mu$.
Proposition \ref{pr:topConj} shows that both smooth and piecewise-smooth maps
could be used as a normal form when equivalence is defined by the existence of a local topological conjugacy.

Proposition \ref{pr:topConj} also shows that the slopes $s_L$ and $s_R$ in \eqref{eq:pwl}
could be taken to be {\em any} values satisfying $0 < s_R < 1 < s_L$.
This runs counter to the conventional construction of the border-collision normal form
that sets $s_L$ and $s_R$, or in higher dimensions the left and right Jacobian matrices,
equal to those of the general map evaluated at the bifurcation \cite{Si16}.
Hence differentiable conjugacies provide one way to justify a particular choice for an (extended) normal form.

\section{Differentiable conjugacies for smooth maps}
\label{sec:diffConj}

Here we construct differentiable conjugacies for smooth monotone maps.
In \S\ref{sec:construction} these conjugacies will be applied to the individual pieces of our piecewise-smooth maps.

We first state a result of Sternberg that in a neighbourhood of a hyperbolic fixed point a smooth map
is differentiably conjugate to a linear map.
We then use orbits of the maps to extend the domain of this conjugacy outwards,
using different constructions for increasing maps and for decreasing maps.
We also show how a differentiable conjugacy between two smooth maps can be found that in addition maps
a given point of one map to a given point of the other map.
This will be used in \S\ref{sec:construction} to obtain a conjugacy
that maps the non-differentiable point of one piecewise-smooth map
to the non-differentiable point of another piecewise-smooth map.
Finally in this section we provide a lemma
that underpins our later demonstration that $M(\mu) \approx N$ in Theorems \ref{th:noBif}--\ref{th:pdLike}.

\subsection{Linearisations}
\label{sub:linearisations}

Throughout this section $(a,b)$ is an open interval containing $0$.
The following result is taken from Sternberg \cite{St57b}.

\begin{lemma}[local linearisation]
Let $f : (a,b) \to \mathbb{R}$ be $C^2$.
Suppose $0$ is a fixed point of $f$ with $\frac{d f}{d x}(0) = \lambda \notin \{ -1,0,1 \}$.
Define the linear map $g(y) = \lambda y$.
Then there exist open neighbourhoods $U \subset (a,b)$ of $0$ and $V \subset \mathbb{R}$ of $0$
such that $f$ on $U$ and $g$ on $V$ are differentiably conjugate.
\label{le:Sternberg}
\end{lemma}

Lemma \ref{le:Sternberg} shows there exists a diffeomorphism $h$ such that the conjugacy relation
\begin{equation}
h(f(x)) = g(h(x))
\label{eq:conjugacyRelation2}
\end{equation}
is satisfied on a neighbourhood of $0$.
We now show that the domain of $h$ can be expanded to any interval on which $f$ is strictly increasing or strictly decreasing
and contains no other fixed points or period-two solutions.
We first treat the increasing case for which the construction
is illustrated in Fig.~\ref{fig:pwsDiffConjBifProof}.

\begin{figure}[h!]
\begin{center}
\includegraphics[height=6cm]{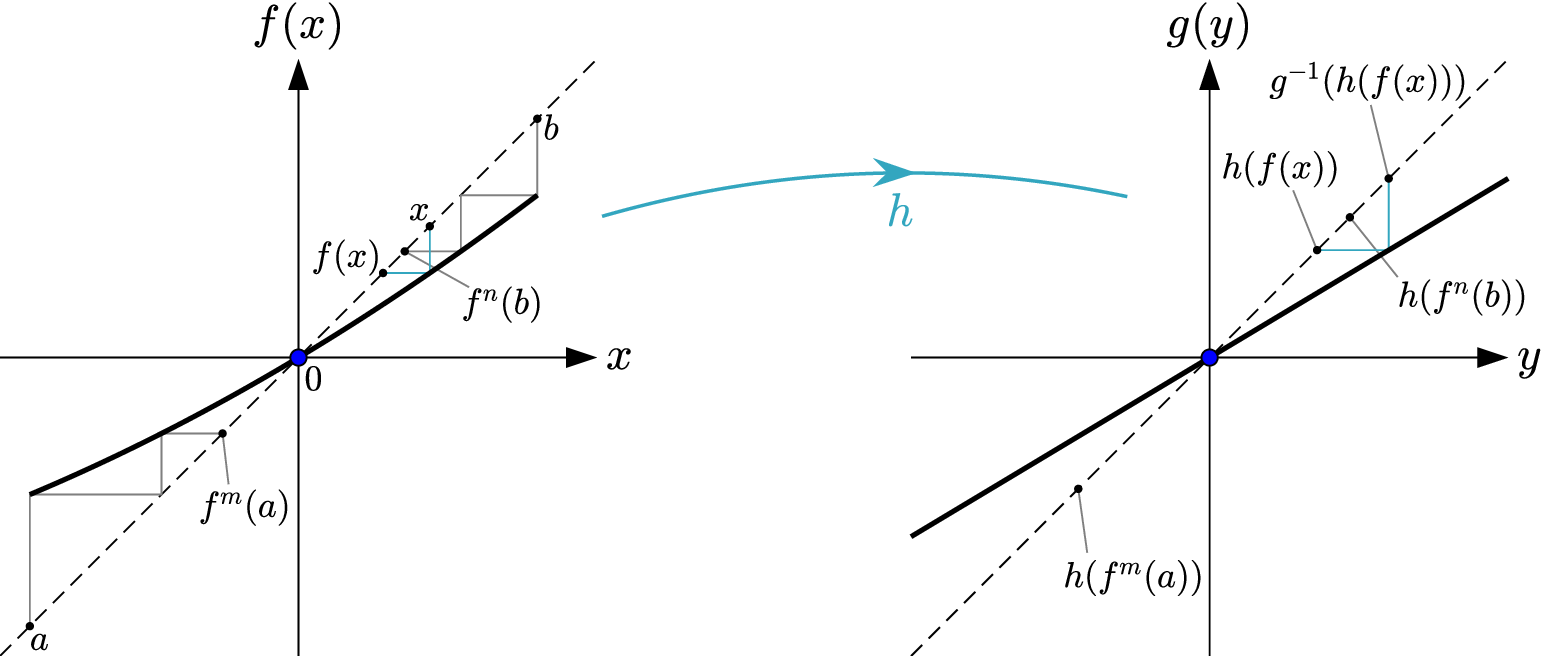} 
\caption{
A sketch showing how the domain of $h$ is extended in the proof of Lemma \ref{le:globalLinearisationInc}.
Here $m = n = 2$ and the sample point $x$ belongs to $[f^n(b),f^{n-1}(b))$.
The corresponding point $h(x)$, defined by \eqref{eq:globalLinearisationIncProof2},
is the result of mapping $x$ under $f$, then $h$, and then $g^{-1}$ as indicated in blue.
\label{fig:pwsDiffConjBifProof}
} 
\end{center}
\end{figure}

\begin{lemma}[global linearisation, increasing case]
Let $f : (a,b) \to \mathbb{R}$ be $C^2$ with
$\frac{d f}{d x}(x) > 0$ for all $x \in (a,b)$.
Suppose $0$ is the unique fixed point of $f$ with $\frac{d f}{d x}(0) = \lambda \ne 1$.
Define the linear map $g(y) = \lambda y$.
Then there exists increasing $h$ that differentiably conjugates $f$ on $(a,b)$ to $g$ on an open interval containing $0$.
\label{le:globalLinearisationInc}
\end{lemma}

\begin{proof}
Suppose $\lambda \in (0,1)$ (the case $\lambda > 1$ can be proved similarly using backwards iterates instead of forward iterates).
Then $f^i(a)$ is an increasing sequence converging to $0$, while $f^i(b)$ is a decreasing sequence converging to $0$.
By Lemma \ref{le:Sternberg} there exist $m, n \ge 0$,
a neighbourhood $V \subset \mathbb{R}$ of $0$,
and a $C^1$ conjugacy $h : \left( f^m(a), f^n(b) \right) \to V$ such that
the conjugacy relation \eqref{eq:conjugacyRelation2} holds
for all $x \in \left( f^m(a), f^n(b) \right)$.
In view of the substitution $y \mapsto -y$ we can assume $h$ is increasing.

We now extend the domain of $h$ to the right until reaching $b$ (assuming $n \ne 0$ otherwise this is not needed)
in such a way that $h$ is $C^1$ and satisfies \eqref{eq:conjugacyRelation2} on the larger domain.
For all $x \in [f^n(b),f^{n-1}(b))$, define
\begin{equation}
h(x) = g^{-1}(h(f(x))),
\label{eq:globalLinearisationIncProof2}
\end{equation}
see Fig.~\ref{fig:pwsDiffConjBifProof}.
This is well-defined, because $f(x)$ belongs to the domain of $h$, while $g^{-1}$ is defined on $\mathbb{R}$.
By construction $h$ satisfies \eqref{eq:conjugacyRelation2} on the larger domain $\left( f^m(a), f^{n-1}(b) \right)$.
Certainly $h$ is $C^1$ on $\left( f^n(b), f^{n-1}(b) \right)$, so it remains to show $h$ is also $C^1$ at $x = f^n(b)$.
By \eqref{eq:conjugacyRelation2} equation \eqref{eq:globalLinearisationIncProof2}
holds for all $x \in \left( 0, f^n(b) \right)$,
so it holds on a neighbourhood of $f^n(b)$.
At $x = f^n(b)$ the right-hand side of \eqref{eq:globalLinearisationIncProof2} is $C^1$
because $h$ is $C^1$ at $f^{n+1}(b)$,
so $h$ is indeed $C^1$ at $f^n(b)$.

By same argument we can extend the domain of $h$ outwards to $f^{n-2}(b)$, then to $f^{n-3}(b)$, and so on until reaching $b$.
Similarly we can extend the domain of $h$ to the left until reaching $a$, and thus covering $(a,b)$.
\end{proof}

The next result accommodates the decreasing case.
The assumption $b < f(a)$ is chosen without loss of generality.

\begin{lemma}[global linearisation, decreasing case]
Let $f : (a,b) \to \mathbb{R}$ be $C^2$
with $\frac{d f}{d x}(x) < 0$ for all $x \in (a,b)$.
Suppose $0$ is a fixed point of $f$ with $\frac{d f}{d x}(0) = \lambda \ne -1$ and $f^2$ has no other fixed points in $(a,b)$.
Suppose $b < f(a)$.
Define the linear map $g(y) = \lambda y$.
Then there exists increasing $h$ that differentiably conjugates
$f$ on $\left( f^{-1}(b), b \right)$ to $g$ on an open interval containing $0$.
\label{le:globalLinearisationDec}
\end{lemma}

\begin{proof}
Suppose $\lambda \in (-1,0)$ (the case $\lambda < -1$ can be proved similarly).
Notice $f^{-1}(b) \in (a,0)$ and $f^i(b) \in (a,b)$ for all $i \ge 0$, with $f^i(b) > 0$ if $i$ is even and $f^i(b) < 0$ if $i$ is odd.
Importantly $f^i(b) \to 0$ as $i \to \infty$ by the assumption that $f^2$ has no other fixed points.
Thus by Lemma \ref{le:Sternberg} there exists even $n \ge 0$,
a neighbourhood $V \subset \mathbb{R}$ of $0$,
and a $C^1$ conjugacy $h : \left( f^{n+1}(b), f^n(b) \right) \to V$ such that
the conjugacy relation \eqref{eq:conjugacyRelation2} holds
for all $x \in \left( f^{n+1}(b), f^n(b) \right)$,
and in view of the substitution $y \mapsto -y$ we can assume $h$ is increasing.

If $n \ne 0$ we extend the domain of $h$ outwards.
For all $x \in \left( f^{n-1}(b), f^{n+1}(b) \right]$, define
\begin{equation}
h(x) = g^{-1}(h(f(x))),
\label{eq:globalLinearisationDecProof2}
\end{equation}
so then $h$ is a $C^1$ conjugacy from $f$ to $g$ on the larger domain
$\left( f^{n-1}(b), f^n(b) \right)$ for the same reasons as in the proof of Lemma \ref{le:globalLinearisationInc}.
By the same argument we can extend the domain to
$\left( f^{n-1}(b), f^{n-2}(b) \right)$,
then to $\left( f^{n-3}(b), f^{n-2}(b) \right)$, and so on until reaching $\left( f^{-1}(b), b \right)$.
\end{proof}

\subsection{Constructing conjugacies that map one given point to another}
\label{sub:matchingPoints}

By two applications of the above lemmas
we can conclude that two smooth maps $f$ and $g$ are differentiably conjugate
near fixed points having the same stability multiplier.
We now show that the corresponding conjugacy can be chosen to map a given point near the fixed point of $f$ to a given point near the fixed point of $g$.
Specifically, the results below give $h(b) = d$,
or more precisely $\lim_{x \to b^-} h(x) = d$.
We cannot also demand $h(a)$ maps to another given point $c$.
For this reason Lemma \ref{le:matchTwoPointsInc} assumes the domain of $g$ extends to $-\infty$ and
constructs the value of $c = h(a)$ over the course of the proof.
This is illustrated in Fig.~\ref{fig:pwsDiffMatching}.

\begin{figure}[h!]
\begin{center}
\includegraphics[height=6cm]{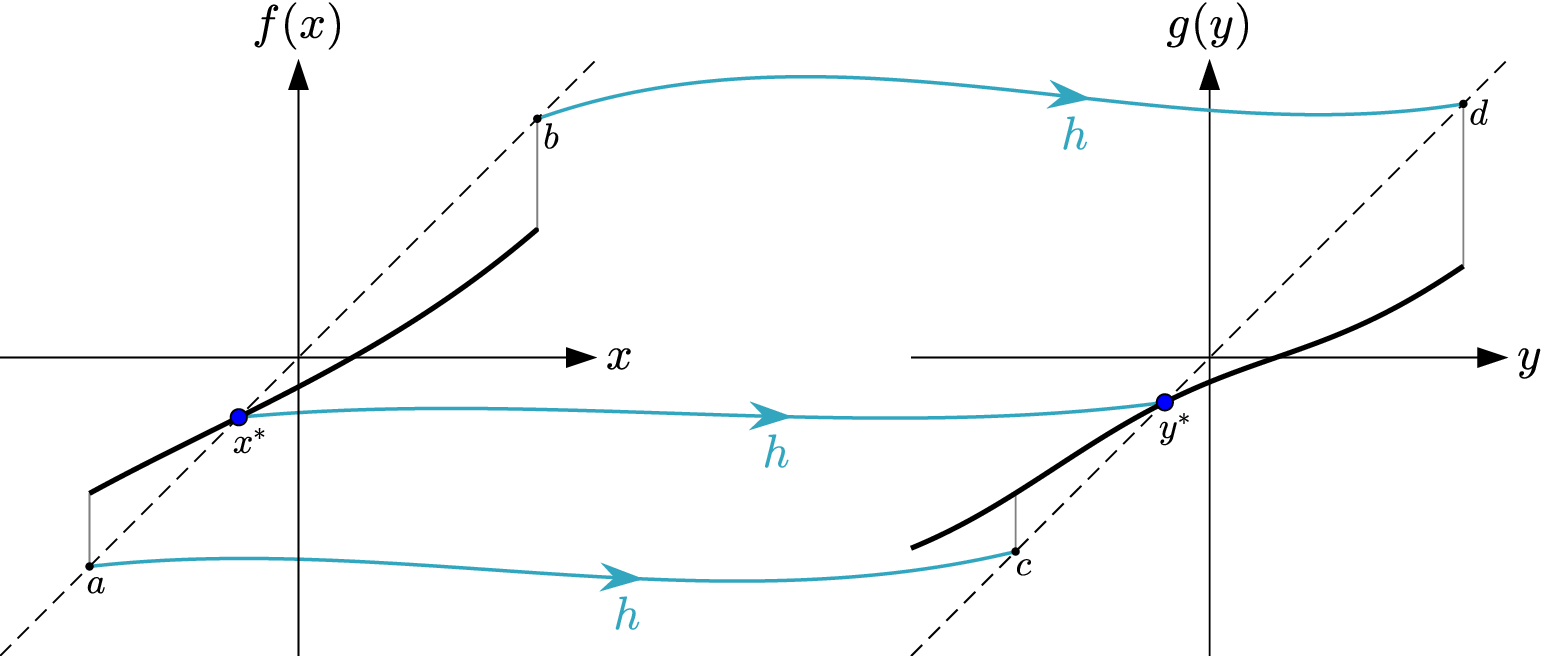} 
\caption{
A sketch of maps $f$ and $g$ satisfying the assumptions of Lemma \ref{le:matchTwoPointsInc}
with $\lambda = f'(x^*) = g'(y^*) \in (0,1)$.
The differentiable conjugacy $h$ maps the points $a$, $x^*$, and $b$
to $c$, $y^*$, and $d$, as indicated.
\label{fig:pwsDiffMatching}
} 
\end{center}
\end{figure}

\begin{lemma}[matching points, increasing case]
Let $f : (a,b) \to \mathbb{R}$ be $C^2$ with $\frac{d f}{d x}(x) > 0$ for all $x \in (a,b)$,
and suppose $f$ has a unique fixed point $x^* \in (a,b)$ with $\frac{d f}{d x}(x^*) = \lambda \ne 1$.
Let $g : (-\infty,d) \to \mathbb{R}$ be $C^2$ with $\frac{d g}{d y}(y) > 0$ for all $y < d$,
and suppose $g$ has a unique fixed point $y^* < d$ with $\frac{d g}{d y}(y^*) = \lambda$.
Also suppose $g(y) \to -\infty$ as $y \to -\infty$.
Then there exists $c < y^*$ and an increasing diffeomorphism
$h : (a,b) \to (c,d)$ that conjugates $f$ to $g$.
\label{le:matchTwoPointsInc}
\end{lemma}

\begin{proof}
Define the linear map $\ell(z) = \lambda z$.
By Lemma \ref{le:globalLinearisationInc} to applied $f$ (shifted so its fixed point is $0$),
there exist $p < 0$, $q > 0$, and an increasing diffeomorphism
$h_1 : (a,b) \to (p,q)$ that conjugates $f$ to $\ell$.
By Lemma \ref{le:globalLinearisationInc} to applied $g$,
for any $\tilde{c} < y^*$ there exist $r < 0$, $s > 0$, and an increasing diffeomorphism
$h_2 : (\tilde{c},d) \to (r,s)$ that conjugates $g$ to $\ell$.
Since $g(y) \to -\infty$ as $y \to -\infty$, we must have $r \to -\infty$ as $\tilde{c} \to -\infty$.

Define the increasing linear function $h_3(z) = \frac{s z}{q}$.
Let $\tilde{c}$ be such that $r < \frac{s p}{q}$.
Then $c = h_2^{-1}(h_3(h_1(a))) = h_2^{-1} \big( \frac{s p}{q} \big)$ is well-defined,
and $h = h_2^{-1} \circ h_3 \circ h_1 : (a,b) \to (c,d)$ is an increasing diffeomorphism that conjugates $f$ to $g$.
\end{proof}

Finally we provide an analogous result for decreasing maps.

\begin{lemma}[matching points, decreasing case]
Let $f : (a,b) \to \mathbb{R}$ be $C^2$ with $\frac{d f}{d x}(x) < 0$ for all $x \in (a,b)$, and suppose $b < f(a)$.
Suppose $x^* \in (a,b)$ is a fixed point of $f$ with $\frac{d f}{d x}(x^*) = \lambda \ne -1$ and $f^2$ has no other fixed points.
Let $g : (c,d) \to \mathbb{R}$ be $C^2$ with $\frac{d g}{d y}(y) < 0$ for all $y \in (c,d)$, and suppose $d < g(c)$.
Suppose $y^* \in (c,d)$ is a fixed point of $g$ with $\frac{d g}{d y}(y^*) = \lambda$ and $g^2$ has no other fixed points.
Then there exists an increasing diffeomorphism
$h : \left( f^{-1}(b), b \right) \to \left( g^{-1}(d), d \right)$ that conjugates $f$ to $g$.
\label{le:matchTwoPointsDec}
\end{lemma}

\begin{proof}
Define the linear map $\ell(z) = \lambda z$.
By Lemma \ref{le:globalLinearisationDec} to applied $f$ 
there exists $q > 0$ and an increasing diffeomorphism $h_1 : (f^{-1}(b),b) \to (\ell^{-1}(q),q)$
that conjugates $f$ to $\ell$.
Similarly there exists $s > 0$ and an increasing diffeomorphism $h_2 : (g^{-1}(d),d) \to (\ell^{-1}(s),s)$
that conjugates $g$ to $\ell$.
Define the increasing linear function $h_3(z) = \frac{s z}{q}$.
Then $h = h_2^{-1} \circ h_3 \circ h_1 : (f^{-1}(b),b) \to (g^{-1}(d),d)$ is an increasing diffeomorphism
that conjugates $f$ to $g$.
\end{proof}

\subsection{A bound on the range of the conjugacy function}
\label{sub:hPrime}

Consider again the scenario of Lemma \ref{le:matchTwoPointsInc}
with $\lambda \in (0,1)$ shown in Fig.~\ref{fig:pwsDiffMatching}.
The points $x^*$, $b$, $y^*$, and $d$ are known {\em a priori};
Lemma \ref{le:matchTwoPointsInc} shows we can construct $h$ that differentiably conjugates $f$ to $g$ with $h(x^*) = y^*$ and $h(b) = d$.

In our later applications of Lemma \ref{le:matchTwoPointsInc}
the difference between $b$ and $x^*$ and the difference between $d$ and $y^*$ are small,
so $f$ on $(x^*,b)$ and $g$ on $(y^*,d)$ are usefully approximated by linear functions, both with slope $\lambda$.
Consequently we can design $h$ to be roughly linear,
in which case its slope anywhere should be approximately $\frac{d - y^*}{b - x^*}$, call this ratio $\chi$.
So for any $x \in (x^*,b)$, and also any $x \in (a,x^*)$ if $x^* - a$ is suitably small,
the ratio $\frac{h(x) - y^*}{x - x^*}$ should be approximately $\chi$.

Lemma \ref{le:bound} (below) makes this claim precise.
Our proof, given in Appendix \ref{app:proof},
is based on the fact for any $x \in (a,b)$ we have $f^n(x) \to x^*$ as $n \to \infty$,
thus $\frac{h(f^n(x)) - x^*}{f^n(x) - x^*} \to h'(x^*)$ as $n \to \infty$.
A bound on the value of $f^n(x)$ in terms of the value of $x$ is obtained from a linear approximation to $f$ using induction on $n$ (Lemma \ref{le:fn}a),
while a similar bound on the value of $h(f^n(x)) = g^n(h(x))$ in terms of the value of $h(x)$ is obtained from a linear approximation to $g$.
Then with $x = b$ we obtain $h'(x^*) \approx \chi$ (with an error proportional to the size of the quadratic terms in $f$ and $g$).

Then we use a linear approximation to $g^{-1}$
to bound the value of $h(x)$ in terms of the value of $g^n(h(x))$ (Lemma \ref{le:fn}b).
This enables us to show $\frac{h(x) - y^*}{x - x^*} \approx h'(x^*)$.
The key conclusion is that the right-hand side of \eqref{eq:bound}
converges to $0$ as the interval $(a,b)$ is shrunk to the point $x^*$
as this enables us later to obtain \eqref{eq:theBound} for any $\delta > 0$.

For brevity Lemma \ref{le:bound} is just stated for the increasing case with $\lambda \in (0,1)$.
With instead $\lambda > 1$ an analogous result can obtained similarly
and given in terms of a bound on the second derivatives of $f^{-1}$ and $g^{-1}$.
With instead $\lambda < 0$ the same type of result holds;
indeed the required calculations are almost identical using $|\lambda|$ in place of $\lambda$.
In Lemma \ref{le:bound} we assume $\chi \le \frac{3}{2}$ to simplify the right-hand side of \eqref{eq:bound}
and because our later applications of Lemma \ref{eq:bound} use $\chi \approx 1$.

\begin{lemma}
Let $f$, $g$, and $h$ be as in Lemma \ref{le:matchTwoPointsInc} with $\lambda \in (0,1)$.
Let $K > 0$ be such that $|f''(x)| \le K$ for all $x \in (a,b)$ and $|g''(y)| \le K$ for $y \in (c,d)$.
Let $r = \frac{\lambda(1-\lambda)}{10 K}$ and suppose ${\rm max}(x^*-a,b-x^*) \le r$.
Let $\chi = \frac{d - y^*}{b - x^*}$ and suppose $\chi \le \frac{3}{2}$.
Then for all $x \in (a,b) \setminus \{ x^* \}$,
\begin{equation}
\left| \frac{h(x) - y^*}{x - x^*} - \chi \right| < \frac{3}{2 r} \big( 4 |x - x^*| + b - x^* \big).
\label{eq:bound}
\end{equation}
\label{le:bound}
\end{lemma}

\section{Differentiable conjugacies for piecewise-smooth maps}
\label{sec:construction}

In this section we construct differentiable conjugacies between piecewise-smooth maps $f$ and $g$.
We assume $f$ is defined on a bounded interval, whereas $g$ is defined on the whole of $\mathbb{R}$.
This allows us to construct $h$ on the entire domain of $f$
without having to worry about escaping the domain of $g$ during the course of the construction.

Specifically we let $a < 0$, $b > 0$, and consider
continuous piecewise-$C^2$ maps $f : (a,b) \to \mathbb{R}$ and $g : \mathbb{R} \to \mathbb{R}$ of the form
\begin{align}
f(x) &= \begin{cases}
f_L(x), & x \le 0, \\
f_R(x), & x \ge 0,
\end{cases} &
g(y) &= \begin{cases}
g_L(y), & y \le 0, \\
g_R(y), & y \ge 0.
\end{cases} &
\label{eq:fg}
\end{align}
For $g$ we assume
\begin{equation}
|g(y)| \to \infty ~\text{as}~ y \to \pm \infty
\label{eq:infty}
\end{equation}
so that we can map backwards under $g_L$ and $g_R$ as many times as we like.

Different cases require different assumptions and employ different constructions of the conjugacy function $h$.
In several cases the construction proceeds as follows.
The map $f$ has a fixed point $x_L < 0$ and we apply 
Lemma \ref{le:matchTwoPointsInc} or Lemma \ref{le:matchTwoPointsDec} to the left pieces of the maps, i.e.~$f_L$ and $g_L$.
Specifically we apply the result to $f_L$ on the interval $(a,0)$ and $g_L$ using $d = 0$.
This produces a differentiable conjugacy $h$ with $h(x_L) = y_L$ and $h(0) = 0$.
We then extend the domain of $h$ outwards by applying
the procedure used in the proof of Lemma \ref{le:globalLinearisationInc}.
Since we have enforced $h(0) = 0$, each step in this procedure
uses either $f_L$ and $g_L$, or $f_R$ and $g_R$.
These are smooth, so $h$ is differentiable on the larger domain,
and a simple application of the chain rule
shows that the non-differentiability of $f$ and $g$ at $0$ does not upset the differentiability of $h$
as long as we assume $f$ and $g$ to have the same ratio of lower and upper derivatives at $0$:
\begin{equation}
\frac{f_L'(0)}{f_R'(0)} = \frac{g_L'(0)}{g_R'(0)}.
\label{eq:equalSlopeRatios}
\end{equation}

\subsection{One fixed point and no period-two solutions}
\label{sub:oneFp}

We begin with the simplest scenario that $f$ and $g$ have exactly one fixed point and no period-two solutions.
Within this scenario we distinguish four cases requiring separate constructions.

We start with the case that $f$ is increasing on $(a,b)$.
Fig.~\ref{fig:pwsDiffConjMap_f} provides a sketch of a map $f$
satisfying the conditions of Lemma \ref{le:constructOneInc}
and for which its fixed point $x_L$ is asymptotically stable, i.e.~$f_L'(x_L) < 1$.
The map $g$ has analogous properties on the whole of $\mathbb{R}$.

\begin{lemma}[increasing case]
Consider piecewise-$C^2$ maps $f$ and $g$ \eqref{eq:fg} satisfying \eqref{eq:infty} and \eqref{eq:equalSlopeRatios}.
Suppose $f_L'(x) > 0$ for all $x \in (a,0]$ and $f_R'(x) > 0$ for all $x \in [0,b)$.
Suppose $f$ has a unique fixed point $x_L \in (a,0)$ with $f_L'(x_L) \ne 1$.
Similarly suppose $g_L'(y) > 0$ for all $y \le 0$ and $g_R'(y) > 0$ for all $y \ge 0$.
Suppose $g$ has a unique fixed point $y_L < 0$ with $g_L'(y_L) = f_L'(x_L)$.
Then $f$ on $(a,b)$ is differentiably conjugate to $g$ on some subset of $\mathbb{R}$.
\label{le:constructOneInc}
\end{lemma}

\begin{figure}[h!]
\begin{center}
\includegraphics[height=6cm]{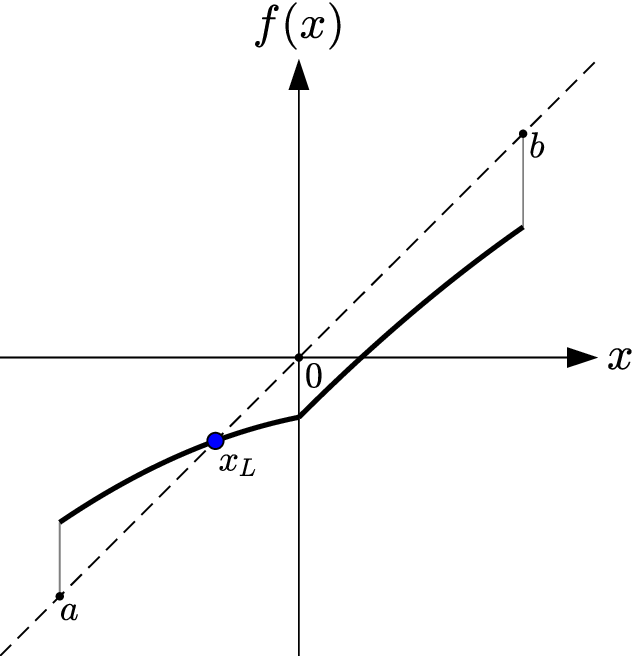} 
\caption{
A piecewise-smooth map $f$ satisfying the conditions of Lemma \ref{le:constructOneInc} with $f_L'(x_L) < 1$.
\label{fig:pwsDiffConjMap_f}
} 
\end{center}
\end{figure}

\begin{proof}
Suppose $f_L'(x_L) \in (0,1)$ (the case $f_L'(x_L) > 1$ can be proved similarly
using backwards iterates instead of forward iterates).
By Lemma \ref{le:matchTwoPointsInc} applied to $f_L$ on the interval $(a,0)$ and $g_L$ using $d = 0$,
there exists $c < y_L$ and an increasing diffeomorphism $h : (a,0) \to (c,0)$ such that
\begin{equation}
h(f_L(x)) = g_L(h(x)),
\label{eq:constructOneIncProof1}
\end{equation}
for all $x \in (a,0)$.
By construction, $\lim_{x \to 0^-} h(x) = 0$, thus $h(f_L(0)) = g_L(0)$ by \eqref{eq:constructOneIncProof1}.
By the continuity of $f$ this is equivalent to $h(f_R(0)) = g_R(0)$.

We now extend the domain of $h$ to $[0,f_R^{-1}(0))$, assuming $f_R^{-1}(0) \le b$ otherwise we only extend it to $[0,b)$.
For all $x \in [0,f_R^{-1}(0))$, define
\begin{equation}
h(x) = g_R^{-1}(h(f_R(x))).
\label{eq:constructOneIncProof2}
\end{equation}
This is well-defined because $f_R(x) \in \left[ f_R(0), 0 \right)$,
so $h(f_R(x))$ is defined with $h(f_R(x)) \ge g_R(0)$ where the inverse of $g_R$ is defined.
Certainly $h$ is continuous and increasing on $[0,f_R^{-1}(0))$,
plus satisfies the conjugacy relation \eqref{eq:conjugacyRelation2} on this interval.
Also it is $C^1$ on $(0,f_R^{-1}(0))$, because $f_R$, $h$, and $g_R^{-1}$ are $C^1$ on the relevant intervals,
so it remains to show $h$ is $C^1$ to $x=0$.
To do this we show that the left and right limits of $h'(x)$ agree at $0$.
From \eqref{eq:constructOneIncProof1}, the chain rule gives
\begin{equation}
h'(x) = \frac{h'(f_L(x)) f_L'(x)}{g_L'(h(x))},
\nonumber
\end{equation}
for $x < 0$. So
\begin{equation}
\lim_{x \to 0^-} h'(x) = \frac{h'(f_L(0)) f_L'(0)}{g_L'(0)},
\label{eq:constructOneIncProof3}
\end{equation}
using $h(0) = 0$.
Similarly \eqref{eq:constructOneIncProof2} gives
\begin{equation}
h'(x) = \frac{h'(f_R(x)) f_R'(x)}{g_R'(h(x))},
\nonumber
\end{equation}
for $x > 0$, so
\begin{equation}
\lim_{x \to 0^+} h'(x) = \frac{h'(f_R(0)) f_R'(0)}{g_R'(0)}.
\label{eq:constructOneIncProof4}
\end{equation}
Notice \eqref{eq:constructOneIncProof3} and \eqref{eq:constructOneIncProof4}
are equal by \eqref{eq:equalSlopeRatios} and the continuity of $f$ at $0$, thus $h$ is $C^1$ at $x=0$.

In the same manner we use \eqref{eq:constructOneIncProof2}
to extend the domain of $h(x)$ to $[f_R^{-1}(0), f_R^{-2}(0))$,
then $[f_R^{-2}(0), f_R^{-3}(0))$, and so on until we reach $b$.
At each step $g_R^{-1}$ is well-defined because we have assumed $g_R(y) \to \infty$ as $y \to \infty$.
The differentiability of $h$ is immediate because this procedure only uses the right pieces of $f$ and $g$,
and we can reach $b$ in a finite number of steps because $f$ has no fixed points between $0$ and $b$.
\end{proof}

Next we treat the decreasing case.
See Fig.~\ref{fig:pwsDiffConjMap_g} for a map $f$ satisfying the conditions of Lemma \ref{le:constructOneDec}.

\begin{lemma}[decreasing case]
Consider piecewise-$C^2$ maps $f$ and $g$ \eqref{eq:fg} satisfying \eqref{eq:infty} and \eqref{eq:equalSlopeRatios}.
Suppose $f_L'(x) < 0$ for all $x \in (a,0]$ and $f_R'(x) < 0$ for all $x \in [0,b)$.
Suppose $x_L < 0$ is a fixed point of $f$ with $f_L'(x_L) \ne -1$
and $f^2$ has no other fixed points.
Also suppose $f(0) > a$ and $f(a) > 0$.
Similarly suppose $g_L'(y) < 0$ for all $y \le 0$ and $g_R'(y) < 0$ for all $y \ge 0$.
Suppose $y_L < 0$ is a fixed point of $g$ with $g_L'(y_L) = f_L'(x_L)$
and $g^2$ has no other fixed points.
Then $f$ on $(a,b)$ is differentiably conjugate to $g$ on some subset of $\mathbb{R}$.
\label{le:constructOneDec}
\end{lemma}

\begin{figure}[h!]
\begin{center}
\includegraphics[height=6cm]{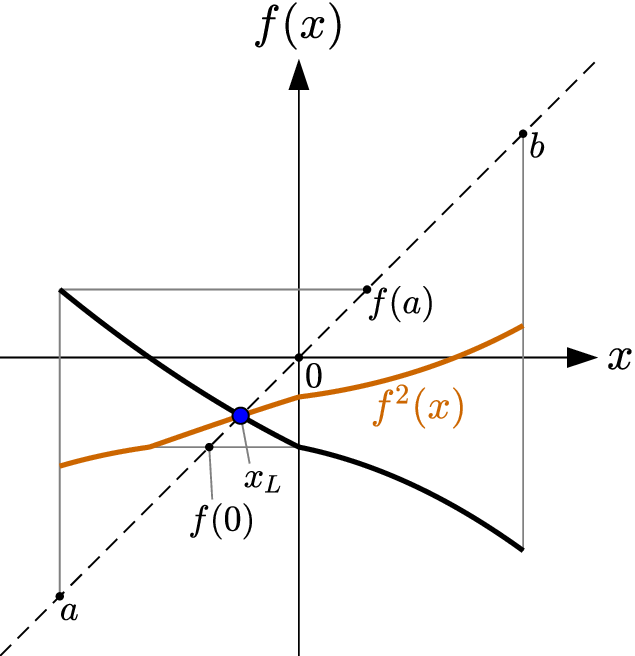} 
\caption{
A piecewise-smooth map $f$ satisfying the conditions of Lemma \ref{le:constructOneDec} with $f_L'(x_L) > -1$.
We also show the second iterate $f^2$.
\label{fig:pwsDiffConjMap_g}
} 
\end{center}
\end{figure}

\begin{proof}
Suppose $f_L'(x_L) \in (-1,0)$ (the case $f_L'(x_L) < -1$ can be proved similarly).
The assumption $f(a) = f_L(a) > 0$
implies $f^{-1}(0) = f_L^{-1}(0)$ is well-defined with $f_L^{-1}(0) > a$, see Fig.~\ref{fig:pwsDiffConjMap_g}.
Also $f(0) = f_L(0) > f_L^{-1}(0)$ due to the assumption on $f^2$.
By Lemma \ref{le:matchTwoPointsDec} applied to $f_L$ and $g_L$
there exists an increasing diffeomorphism $h : \left( f_L^{-1}(0), 0 \right) \to \left( g_L^{-1}(0),0 \right)$
such that \eqref{eq:constructOneIncProof1} holds for all $x \in \left( f_L^{-1}(0), 0 \right)$
(note that $g_L^{-1}(0)$ is well-defined because we have assumed $g_L(y) \to \infty$ as $y \to -\infty$).

We now extend the domain of $h$ to $\left[ 0, f^{-2}(0) \right)$, assuming $f^{-2}(0) \le b$ otherwise we just extend to $b$.
For all $x \in \left[ 0, f^{-2}(0) \right)$ define
\begin{equation}
h(x) = g_R^{-1}(h(f_R(x))).
\label{eq:constructOneDecProof2}
\end{equation}
This is well-defined because as $x$ ranges from $0$ to $f^{-2}(0)$,
$h(f_R(x))$ ranges from $g^{-1}(0)$ to $g(0)$,
and so $h(x)$ ranges from $0$ to $g^{-2}(0)$.
Certainly $h$ is $C^1$ on $\left( 0, f^{-2}(0) \right)$;
it is also $C^1$ at $x = 0$ because as in the previous proof we have \eqref{eq:constructOneIncProof3} and \eqref{eq:constructOneIncProof4}.

In a similar fashion we can extend the domain of $h$ to $\left( f^{-3}(0), f^{-1}(0) \right]$, assuming $f^{-3}(0) > a$ otherwise we just extend to $a$,
by defining $h(x) = g_L^{-1}(h(f_L(x)))$.
Here the differentiability of $h$ at $f^{-1}(0)$ follows from its differentiability at $0$.
We then further extend the domain of $h$ outwards by alternately using the left and right pieces of the maps until reaching $a$ or $b$.
At each step $g_L^{-1}$ and $g_R^{-1}$ are well-defined by \eqref{eq:infty}.
This will be achieved in finitely many steps because $x_L$ is the only fixed point of $f^2$.

Suppose we first reach $a$ (the following arguments work similarly if we first reach $b$).
We then use \eqref{eq:constructOneDecProof2} again to extend
the domain of $h$ to either $b$, if $f(b) \ge a$, or to $f^{-1}(a)$, if $f(b) < a$.
Finally in the latter case we extend $h$ from $f^{-1}(a)$ to $b$
by defining it to be any increasing $C^1$ function.
This is sufficient because $h$ does not need to satisfy the conjugacy relation for $x \in \left[ f^{-1}(a), b \right)$
because $f(x) \notin (a,b)$, see Definition \ref{df:conj}.
In either case we have extended $h$ to $(a,b)$ as required.
\end{proof}

We now treat the non-monotone cases
illustrated in Fig.~\ref{fig:pwsDiffConjMap_h} and \ref{fig:pwsDiffConjMap_i}.

\begin{lemma}[non-monotone, increasing at the fixed point case]
Consider piecewise-$C^2$ maps $f$ and $g$ \eqref{eq:fg} satisfying \eqref{eq:infty} and \eqref{eq:equalSlopeRatios}.
Suppose $f_L'(x) > 0$ for all $x \in (a,0]$ and $f_R'(x) < 0$ for all $x \in [0,b)$.
Suppose $f$ has a unique fixed point $x_L < 0$ with $f_L'(x_L) \in (0,1)$.
Similarly suppose $g_L'(y) > 0$ for all $y \le 0$ and $g_R'(y) < 0$ for all $y \ge 0$.
Suppose $g$ has a unique fixed point $y_L < 0$ with $g_L'(y_L) = f_L'(x_L)$.
Then $f$ on $(a,b)$ is differentiably conjugate to $g$ on some subset of $\mathbb{R}$.
\label{le:constructOneNonMonInc}
\end{lemma}

\begin{figure}[h!]
\begin{center}
\includegraphics[height=6cm]{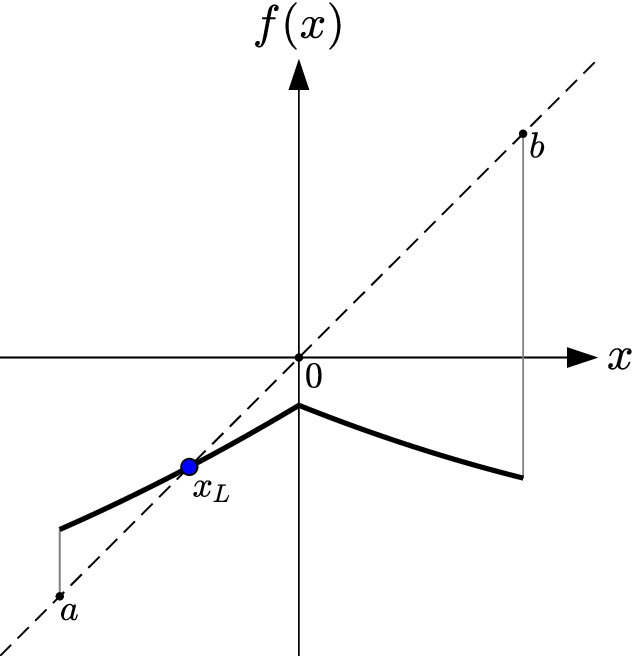} 
\caption{
A piecewise-smooth map $f$ satisfying the conditions of Lemma \ref{le:constructOneNonMonInc}.
\label{fig:pwsDiffConjMap_h}
} 
\end{center}
\end{figure}

\begin{proof}
As in the proof of Lemma \ref{le:constructOneInc}
we construct an increasing diffeomorphism $h : (a,0) \to (c,0)$, for some $c < y_L$,
such that \eqref{eq:constructOneIncProof1} holds for all $x \in (a,0)$.
We then use
\begin{equation}
h(x) = g_R^{-1}(h(f_R(x)))
\label{eq:constructOneNonMonIncProof2}
\end{equation}
to extend the domain of $h$ to $[0,b)$, assuming $f(b) \ge a$ otherwise we only extend to $f_R^{-1}(a)$.
This is well-defined because $f_R(x) \in (a,f(0)]$, where $h$ is already defined,
and $h(f_R(x)) \in (c,g(0)]$, where $g_R^{-1}$ is well-defined.
Clearly $h$ is continuous, increasing, and $C^1$, except possibly at $x=0$.
But it is $C^1$ here too because as in the proof of Lemma \ref{le:constructOneInc} we have
\eqref{eq:constructOneIncProof3} and \eqref{eq:constructOneIncProof4}.

In the case $f(b) > a$ we lastly extend $h$ from $f_R^{-1}(a)$ to $b$
by defining it to be some increasing $C^1$ function (as in the previous proof).
\end{proof}

\begin{lemma}[non-monotone, decreasing at the fixed point case]
Consider piecewise-$C^2$ maps $f$ and $g$ \eqref{eq:fg} satisfying \eqref{eq:infty} and \eqref{eq:equalSlopeRatios}.
Suppose $f_L'(x) < 0$ for all $x \in (a,0]$ and $f_R'(x) > 0$ for all $x \in [0,b)$.
Suppose $x_L < 0$ is a fixed point of $f$ with $f_L'(x_L) \in (-1,0)$
and $f^2$ has no other fixed points.
Also suppose $f(0) > a$ and $f(a) > 0$.
Similarly suppose $g_L'(y) < 0$ for all $y \le 0$ and $g_R'(y) > 0$ for all $y \ge 0$.
Suppose $y_L < 0$ is a fixed point of $g$ with $g_L'(y_L) = f_L'(x_L)$
and $g^2$ has no other fixed points.
Then $f$ on $(a,b)$ is differentiably conjugate to $g$ on some subset of $\mathbb{R}$.
\label{le:constructOneNonMonDec}
\end{lemma}

\begin{figure}[h!]
\begin{center}
\includegraphics[height=6cm]{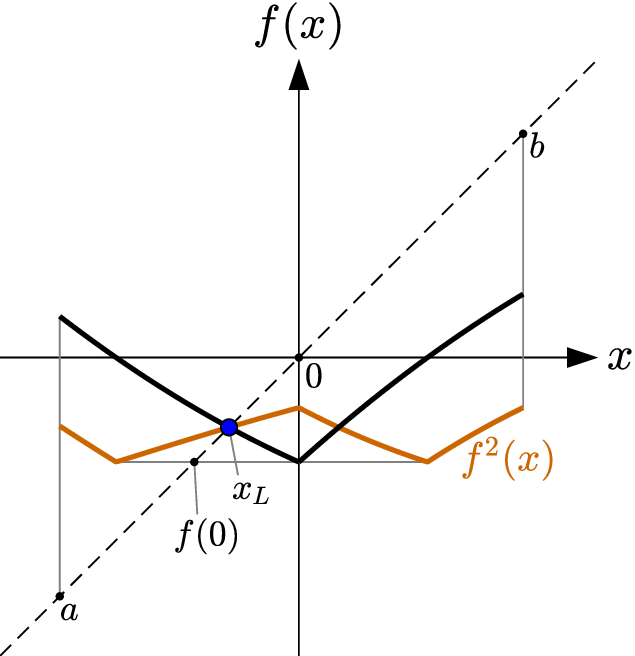} 
\caption{
A piecewise-smooth map $f$ satisfying the conditions of Lemma \ref{le:constructOneNonMonDec} and its second iterate $f^2$.
\label{fig:pwsDiffConjMap_i}
} 
\end{center}
\end{figure}

\begin{proof}
As in the proof of Lemma \ref{le:constructOneDec}
we construct an increasing diffeomorphism $h : \left( f_L^{-1}(0), 0 \right) \to \left( g_L^{-1}(0), 0 \right)$
such that \eqref{eq:constructOneIncProof1} holds for all $x \in \left( f_L^{-1}(0), 0 \right)$.
We then use
\begin{equation}
h(x) = g_R^{-1}(h(f_R(x)))
\label{eq:constructOneNonMonDecProof2}
\end{equation}
to extend the domain of $h$ to $\left[ 0, f_R^{-1}(0) \right)$, assuming $f(b) > 0$ otherwise we only extend it to $b$.
This is well-defined
with $h(x)$ ranging from $0$ to $g_R^{-1}(0)$ as $x$ ranges from $0$ to $f_R^{-1}(0)$.
As above $h$ is continuous, increasing, and $C^1$ because
\eqref{eq:constructOneIncProof3} and \eqref{eq:constructOneIncProof4},
and as in the proof of Lemma \ref{le:constructOneInc}
we can extend the domain of $h$ to the right until reaching $b$.

It remains to extend the domain of $h$ to the left until reaching $a$.
If $f(a) \le b$ we achieve this defining
\begin{equation}
h(x) = g_L^{-1}(h(f_L(x))),
\label{eq:constructOneNonMonDecProof3}
\end{equation}
for all $x \in \left( a, f_L^{-1}(0) \right]$.
If instead $f(a) > b$ we only do this for $x \in \left( f_L^{-1}(b), f_L^{-1}(0) \right]$,
and then finally define $h$ on $\left( a, f_L^{-1}(b) \right]$
to be some increasing $C^1$ function.
\end{proof}

\subsection{Two fixed points}
\label{sub:twoFps}

There are two cases for us to consider in which $f$ and $g$ have two fixed points,
a monotone case, Fig.~\ref{fig:pwsDiffConjMap_j}, and a non-monotone case, Fig.~\ref{fig:pwsDiffConjMap_k}.
In neither case do we obtain a single differentiable conjugacy $h$ on the whole interval $(a,b)$
because we cannot in general extend a conjugacy through a fixed point while maintaining differentiability.

\begin{lemma}[monotone case]
Consider piecewise-$C^2$ maps $f$ and $g$ \eqref{eq:fg} satisfying \eqref{eq:infty} and \eqref{eq:equalSlopeRatios}.
Suppose $f_L'(x) > 0$ for all $x \in (a,0]$ and $f_R'(x) > 0$ for all $x \in [0,b)$.
Suppose $f$ has exactly two fixed points, $x_L < 0$ with $f_L'(x_L) > 1$,
and $x_R > 0$ with $f_R'(x_R) \in (0,1)$.
Similarly suppose $g_L'(y) > 0$ for all $y \le 0$ and $g_R'(y) > 0$ for all $y \ge 0$.
Suppose $g$ has exactly two fixed points, $y_L < 0$ with $g_L'(y_L) = f_L'(x_L)$,
and $y_R > 0$ with $g_R'(y_R) = f_R'(x_R)$.
Then $f$ on $(a,x_R)$ is differentiably conjugate to $g$ on $(c,y_R)$, for some $c < y_L$,
and $f$ on $(x_L,b)$ is differentiably conjugate to $g$ on $(y_L,d)$, for some $d > y_R$.
\label{le:constructTwoMon}
\end{lemma}

\begin{figure}[h!]
\begin{center}
\includegraphics[height=6cm]{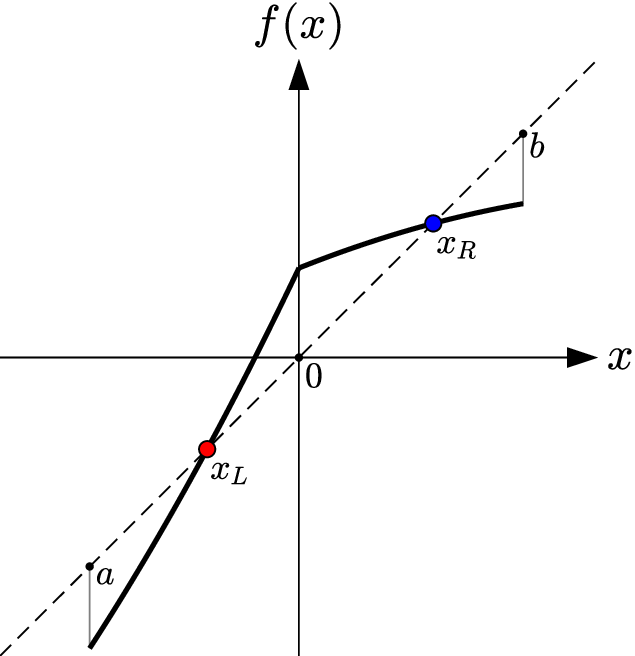} 
\caption{
A piecewise-smooth map $f$ satisfying the conditions of Lemma \ref{le:constructTwoMon}.
\label{fig:pwsDiffConjMap_j}
} 
\end{center}
\end{figure}

\begin{proof}
We just demonstrate the second conjugacy as the first can be constructed similarly
using backwards iterates instead of forwards iterates.
To the map $f_R$ on the interval $(0,b)$ we apply Lemma \ref{le:matchTwoPointsInc},
except with the reverse orientation (i.e.~$x \mapsto -x$).
This shows there exists $d > y_R$ and an increasing diffeomorphism $h : (0,b) \to (0,d)$ such that
\begin{equation}
h(f_R(x)) = g_R(h(x)),
\label{eq:constructTwoMonProof1}
\end{equation}
for all $x \in (0,b)$.
We now use
\begin{equation}
h(x) = g_L^{-1}(h(f_L(x))),
\label{eq:constructTwoMonProof2}
\end{equation}
to extend the domain of $h$ to $\left( f_L^{-1}(0), 0 \right]$,
then $\left( f_L^{-2}(0), f_L^{-1}(0) \right]$, and so on, indefinitely.
As above this produces a differentiable function $h$,
and since $f_L^{-n}(0) \to x_L$ and $g_L^{-n}(0) \to y_L$,
the result is a diffeomorphism $h : (x_L,b) \to (y_L,d)$
satisfying the conjugacy relation \eqref{eq:conjugacyRelation2}.
\end{proof}

\begin{lemma}[non-monotone case]
Consider piecewise-$C^2$ maps $f$ and $g$ \eqref{eq:fg} satisfying \eqref{eq:infty} and \eqref{eq:equalSlopeRatios}.
Suppose $f_L'(x) > 0$ for all $x \in (a,0]$ and $f_R'(x) < 0$ for all $x \in [0,b)$.
Suppose $f$ has exactly two fixed points, $x_L < 0$ with $f_L'(x_L) > 1$,
and $x_R > 0$ with $f_R'(x_R) \in (0,1)$,
and $f^2$ has no other fixed points.
Also suppose $f(b) < x_L$.
Similarly suppose $g_L'(y) > 0$ for all $y \le 0$ and $g_R'(y) < 0$ for all $y \ge 0$.
Suppose $g$ has exactly two fixed points, $y_L < 0$ with $g_L'(y_L) = f_L'(x_L)$,
and $y_R > 0$ with $g_R'(y_R) = f_R'(x_R)$, and $g^2$ has no other fixed points.
Then $f$ on $(a,0)$ is differentiably conjugate to $g$ on $(c,0)$, for some $c > y_L$,
and $f$ on $\left( x_L, f_R^{-1}(x_L) \right)$ is differentiably conjugate to $g$ on $\left( y_L, g_R^{-1}(y_L) \right)$.
\label{le:constructTwoNonMon}
\end{lemma}

\begin{figure}[h!]
\begin{center}
\includegraphics[height=6cm]{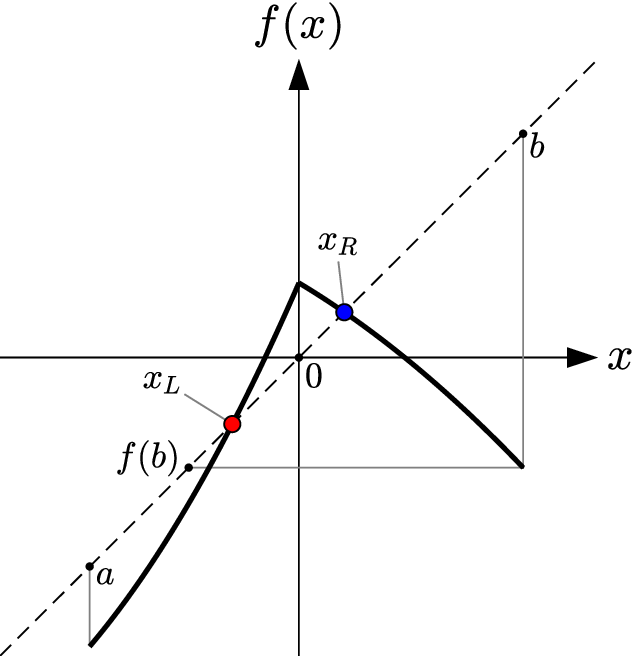} 
\caption{
A piecewise-smooth map $f$ satisfying the conditions of Lemma \ref{le:constructTwoNonMon}.
\label{fig:pwsDiffConjMap_k}
} 
\end{center}
\end{figure}

\begin{proof}
The first conjugacy concerns only $f_L$ and $g_L$, which are smooth and increasing,
so this result follows directly from Lemma \ref{le:matchTwoPointsInc}.

To obtain the second conjugacy
we apply Lemma \ref{le:matchTwoPointsDec} (with $x \mapsto -x$)
to $f_R$ on $\left( 0, f_R^{-1}(0) \right)$ and $g_R$ on $\left( 0, g_R^{-1}(0) \right)$.
This gives an increasing diffeomorphism $h : \left( 0, f_R^{-1}(0) \right) \to \left( 0, g_R^{-1}(0) \right)$
such that \eqref{eq:constructTwoMonProof1} holds for all $x \in \left( 0, f_R^{-1}(0) \right)$.
Notice this application of Lemma \ref{le:matchTwoPointsDec}
relies on the assumption that $f^2$ and $g^2$ have no other fixed points.

We then use \eqref{eq:constructTwoMonProof2}
to extend the domain of $h$ to $\left( f_L^{-1}(0), 0 \right]$,
then $\left( f_L^{-2}(0), f_L^{-1}(0) \right]$, and so on,
resulting in $h$ defined on $\left( x_L, f_R^{-1}(0) \right)$.
Lastly we use
\begin{equation}
h(x) = g_R^{-1}(h(f_R(x))),
\label{eq:constructTwoNonMonProof1}
\end{equation}
to further extend the domain to $\left[ f_R^{-1}(0), f_R^{-1}(x_L) \right)$.
Notice $f_R^{-1}(x_L)$ is well-defined by the assumption $f(b) < x_L$,
and $g_R^{-1}(y)$ is well-defined for all $y \in (y_L, 0)$ by the assumption $g_R(y) \to -\infty$ as $y \to \infty$.
The end result is $h : \left( x_L, f_R^{-1}(x_L) \right) \to \left( y_L, g_R^{-1}(y_L) \right)$
satisfying the conjugacy relation \eqref{eq:conjugacyRelation2} by construction,
and $h$ is a diffeomorphism for the same reasons as in the earlier proofs.
\end{proof}

\subsection{No fixed points}
\label{sub:noFps}

Here we treat cases for which $f$ and $g$ have no fixed points,
Figs.~\ref{fig:pwsDiffConjMap_l} and \ref{fig:pwsDiffConjMap_m}.
These require a different approach to the above scenarios
because there are no fixed points about which we can use Sternberg's linearisation to initiate the construction.
In the monotone case with $f(0) < 0$ we instead pick $z \in (f(0),0)$ and define $h$ on the interval $I = (z,0)$ to be the identity map.
We then propagate $h$ to $f(I)$ via $h(x) = g \left( h \left( f^{-1}(x) \right) \right)$;
importantly $I \cap f(I) = \varnothing$,
and in the gap between $I$ and $f(I)$ we define $h$ by inserting a $C^1$ jump function.
We then propagate $h$ to the entire domain $(a,b)$.

A similar construction was used for smooth maps in \cite{GlSi23}, see Theorem 2.4.
In the piecewise-smooth setting we just also need $h(0) = 0$
so that $h$ is differentiable on $(a,b)$
as a consequence of the assumption \eqref{eq:equalSlopeRatios} of equal ratios of derivatives.

\begin{lemma}[monotone case]
Consider piecewise-$C^2$ maps $f$ and $g$ \eqref{eq:fg} satisfying \eqref{eq:infty} and \eqref{eq:equalSlopeRatios}.
Suppose $f_L'(x) > 0$ for all $x \in (a,0]$ and $f_R'(x) > 0$ for all $x \in [0,b)$.
Suppose $f$ has no fixed points and $a < f(0) < 0$.
Similarly suppose $g_L'(y) > 0$ for all $y \le 0$ and $g_R'(y) > 0$ for all $y \ge 0$.
Suppose $g$ has no fixed points and $g(0) < 0$.
Then $f$ on $(a,b)$ is differentiably conjugate to $g$ on some subset of $\mathbb{R}$.
\label{le:constructZeroMon}
\end{lemma}

\begin{figure}[h!]
\begin{center}
\includegraphics[height=6cm]{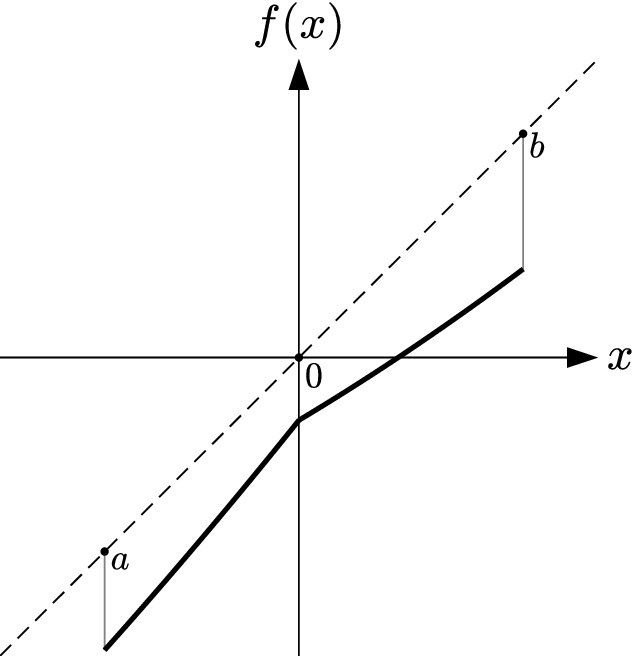} 
\caption{
A piecewise-smooth map $f$ satisfying the conditions of Lemma \ref{le:constructZeroMon}.
\label{fig:pwsDiffConjMap_l}
} 
\end{center}
\end{figure}

\begin{proof}
Let $z \in (f(0),0)$ be such that $f(z) \ge a$.
Let $I = (z,0)$ and define
\begin{equation}
h(x) = x,
\label{eq:constructZeroMonProof1}
\end{equation}
for all $x \in I$.
Next define
\begin{equation}
h(x) = g_L \left( h \left( f_L^{-1}(x) \right) \right),
\label{eq:constructZeroMonProof2}
\end{equation}
for all $x \in f(I) = \left( f(z), f(0) \right)$.
Notice $I \cap f(I) = \varnothing$ and that $h$ is increasing and $C^1$ on $I$ and on $f(I)$.
On the gap $\left[ f(0), z \right]$ between $I$ and $f(I)$
insert a $C^1$ jump function so that $h$ is $C^1$ on $\left( f(z), 0 \right)$.
In particular, by differentiating \eqref{eq:constructZeroMonProof2},
\begin{equation}
h'(f(0))
= \lim_{x \to 0^-} h'(f(x))
= \lim_{x \to 0^-} \frac{g_L'(h(x)) h'(x)}{f_L'(x)}
= \frac{g_L'(0)}{f_L'(0)}.
\label{eq:constructZeroMonProof10}
\end{equation}
where we have used the fact that the derivative of \eqref{eq:constructZeroMonProof1} is $1$.

For each $k = 1,2,\ldots$ until reaching $a$, we define $h(x)$ on $\left( f^{k+1}(z), f^k(z) \right]$
by \eqref{eq:constructZeroMonProof2}.
Notice $h$ is $C^1$ on $(a,0)$ because $h$ is $C^1$ on $\left( f(z), z \right]$ and the
right-hand side of \eqref{eq:constructZeroMonProof2} is smooth.
For each $k = 1,2,\ldots$ until reaching $b$, we define $h(x)$ on $\left[ f^{-(k-1)}(0), f^{-k}(0) \right)$ by
\begin{equation}
h(x) = g_R^{-1} \left( h(f_R(x)) \right).
\label{eq:constructZeroMonProof3}
\end{equation}
The resulting function $h$ defined on $(a,b)$
satisfies the conjugacy relation \eqref{eq:conjugacyRelation2}
for all $x \in (a,b)$ for which $f(x) \in (a,b)$.

It remains to show $h$ is $C^1$ on $[0,b)$,
and this is trivial except at $0$ and its preimages under $f$.
By \eqref{eq:constructZeroMonProof1},
\begin{equation}
\lim_{x \to 0^-} h'(x) = 1,
\label{eq:constructZeroMonProof20}
\end{equation}
whereas by \eqref{eq:constructZeroMonProof3},
\begin{equation}
\lim_{x \to 0^+} h'(x)
= \frac{h'(f(0)) f_R'(0)}{g_R'(0)}
= \frac{f_R'(0) g_L'(0)}{f_L'(0) g_R'(0)} = 1,
\label{eq:constructZeroMonProof21}
\end{equation}
using \eqref{eq:constructZeroMonProof10} and \eqref{eq:equalSlopeRatios}.
Since \eqref{eq:constructZeroMonProof20} and \eqref{eq:constructZeroMonProof21} agree, $h$ is $C^1$ at $0$.
It is also $C^1$ at all $f^{-k}(0) \in (a,b)$ with $k \ge 1$
because the right-hand side of \eqref{eq:constructZeroMonProof3} is smooth.
\end{proof}

\begin{lemma}[non-monotone case]
Consider piecewise-$C^2$ maps $f$ and $g$ \eqref{eq:fg} satisfying \eqref{eq:infty} and \eqref{eq:equalSlopeRatios}.
Suppose $f_L'(x) > 0$ for all $x \in (a,0]$ and $f_R'(x) < 0$ for all $x \in [0,b)$.
Suppose $f$ has no fixed points and $a < f(0) < 0$.
Similarly suppose $g_L'(y) > 0$ for all $y \le 0$ and $g_R'(y) < 0$ for all $y \ge 0$.
Suppose $g$ has no fixed points and $g(0) < 0$.
Then $f$ on $(a,b)$ is differentiably conjugate to $g$ on some subset of $\mathbb{R}$.
\label{le:constructZeroNonMon}
\end{lemma}

\begin{figure}[h!]
\begin{center}
\includegraphics[height=6cm]{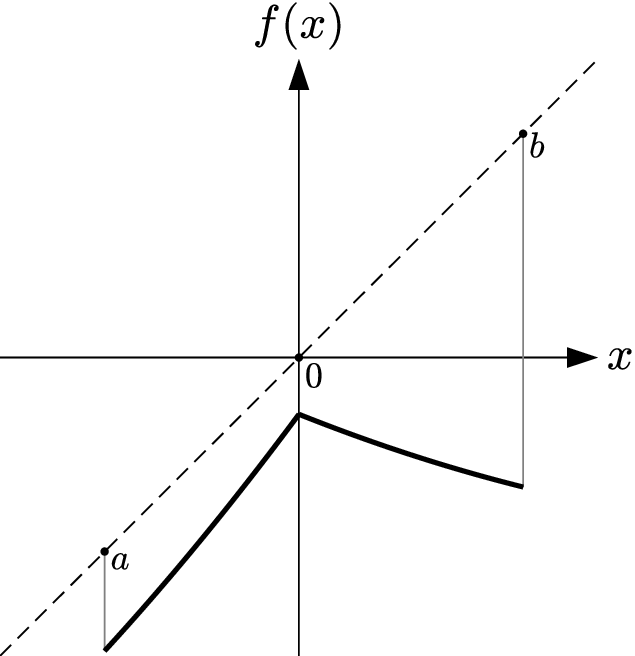} 
\caption{
A piecewise-smooth map $f$ satisfying the conditions of Lemma \ref{le:constructZeroNonMon}.
\label{fig:pwsDiffConjMap_m}
} 
\end{center}
\end{figure}

\begin{proof}
Similar to the previous proof
we can construct an increasing diffeomorphism $h : (a,0) \to (c,0)$, for some $c < 0$,
that conjugates $f$ to $g$.
We can then extend the domain of $h$ to $[0,b)$
as in the proof of Lemma \ref{le:constructOneNonMonInc}
and $h$ is $C^1$ on $(a,b)$ due to \eqref{eq:equalSlopeRatios}.
\end{proof}

\subsection{One fixed point and one period-two solution}
\label{sub:LRcycle}

Here we revisit the case in Lemma \ref{le:constructOneDec} for which both pieces of $f$ are decreasing,
but now suppose $f$ has a period-two solution with one point on each side of $x=0$, Fig.~\ref{fig:pwsDiffConjMap_n}.
This result is used in \S\ref{sec:matching} to prove
Theorem \ref{th:pdLike} for period-doubling-like border-collision bifurcations.

\begin{lemma} 
Consider piecewise-$C^2$ maps $f$ and $g$ \eqref{eq:fg} satisfying \eqref{eq:infty} and \eqref{eq:equalSlopeRatios}.
Suppose $f_L'(x) < 0$ for all $x \in (a,0]$ and $f_R'(x) < 0$ for all $x \in [0,b)$.
Suppose $x_R > 0$ is a fixed point of $f$ with $f_R'(x_R) \ne -1$.
Also suppose 
Suppose $\{ u_L, u_R \}$ is a period-two solution of $f$, with $a < u_L < 0$ and $x_R < u_R < b$
and stability multiplier $\lambda \ne 1$.
Suppose $f^2$ has no fixed points other than $x_L$, $u_L$, and $u_R$.
Similarly suppose $g_L'(y) < 0$ for all $y \le 0$ and $g_R'(y) < 0$ for all $y \ge 0$.
Suppose $y_R > 0$ is a fixed point of $g$ with $g_R'(y_R) = f_R'(x_R)$.
Suppose $\{ v_L, v_R \}$ is a period-two solution of $g$, with $v_L < 0$ and $v_R > y_R$
and stability multiplier $\lambda$.
Suppose $g^2$ has no fixed points other than $y_L$, $v_L$, and $v_R$.
Then $f$ on $(u_L,u_R)$ is differentiably conjugate to $g$ on $(v_L,v_R)$,
and $f$ on $(a,x_R) \cup (x_R,b)$ is differentiably conjugate to $g$ on $(c,y_R) \cup (y_R,d)$
for some $c < v_L$ and $d > v_R$.
\label{le:constructLR}
\end{lemma}

\begin{figure}[h!]
\begin{center}
\includegraphics[height=6cm]{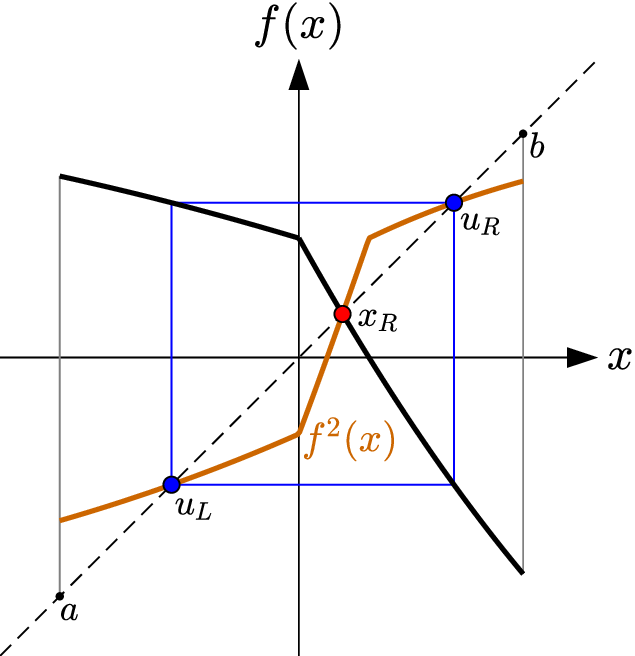} 
\caption{
A piecewise-smooth map $f$ satisfying the conditions of Lemma \ref{le:constructLR} with $f_R'(x_R) < -1$.
\label{fig:pwsDiffConjMap_n}
} 
\end{center}
\end{figure}

\begin{proof}
Suppose $f_R'(x_R) < -1$ as in Fig.~\ref{fig:pwsDiffConjMap_n} (the case $f_R'(x_R) > -1$ can be proved similarly).
In this case $x_R$ is unstable, thus the period-two solution is stable, i.e.~$\lambda \in (0,1)$.

We first construct a conjugacy on $(u_L,u_R)$ by using the fixed point $x_R$.
Unlike earlier proofs where such conjugacies were built from stable fixed points (e.g.~Lemma \ref{le:constructOneDec}),
here the fixed point is unstable so we use the inverse of $f_R$
to ensure the conjugacy relation holds on the entire domain of $h$.
By Lemma \ref{le:matchTwoPointsDec} applied to $f_R^{-1}$ and $g_R^{-1}$,
there exists an increasing diffeomorphism $h : \left( f^2(0), f(0) \right) \to \left( g^2(0), g(0) \right)$ such that
\begin{equation}
h \left( f_R^{-1}(x) \right) = g_R^{-1}(h(x)), 
\nonumber
\end{equation}
for all $x \in \left( f^2(0), f(0) \right)$.
We then use
$h(x) = g_L \left( h \left( f_L^{-1}(x) \right) \right)$
to extend the domain of $h$ to $\left[ f(0), f^3(0) \right)$.
We then use $h(x) = g_R \left( h \left( f_R^{-1}(x) \right) \right)$
to extend the domain to $\left( f^4(0), f^2(0) \right]$,
and so on, indefinitely.
In this way the domain of $h$ is extended to $(u_L,u_R)$, and its range to $(v_L,v_R)$.
As in earlier proofs the constraint \eqref{eq:equalSlopeRatios} implies $h$ is a diffeomorphism,
and by construction $h(f(x)) = g(h(x))$ for all $x \in (u_L,u_R)$.
This shows $f$ on $(u_L,u_R)$ is differentiably conjugate to $g$ on $(v_L,v_R)$.

To construct the second conjugacy, suppose for a moment $a \le f(b)$ and $f(a) \le b$.
We first apply Lemma \ref{le:constructTwoMon} to $f^2$ and $g^2$
(these maps have the same ratio of lower and upper derivatives at $0$
because $f$ and $g$ have the same ratio of lower and upper derivatives at $0$).
This gives us a diffeomorphism $h_1 : (a,x_R) \to (c,y_R)$, for some $c < v_L$, such that
\begin{equation}
h_1 \left( f^2(x) \right) = g^2 \left( h_1(x) \right)
\label{eq:constructLRProof20}
\end{equation}
for all $x \in (a,x_R)$.
Now define
\begin{equation}
h(x) = \begin{cases}
h_1(x), & a < x < x_R \,, \\
g^{-1} \left( h_1 \left( f(x) \right) \right), & x_R < x < b,
\end{cases}
\label{eq:constructLRProof30}
\end{equation}
which exclusively uses $f = f_R$ and $g = g_R$ in the second piece.
This is a diffeomorphism from $(a,x_R) \cup (x_R,b)$ to $(c,y_R) \cup (y_R,d)$, for some $d > v_R$,
and provides a conjugacy from $f$ to $g$ for the following reasons.
For any $x \in (a,x_R)$ we have $f(x) \in (x_R,b)$, so
\begin{equation}
h(f(x)) = g^{-1} \left( h_1 \left( f^2(x) \right) \right) = g \left( h_1(x) \right) = g(h(x)),
\nonumber
\end{equation}
using \eqref{eq:constructLRProof20} for the middle equality,
while for any $x \in (x_R,b)$ we have $f(x) \in (a,x_R)$, so
\begin{equation}
h(f(x)) = h_1(f(x)) = g \left( g^{-1} \left( h_1 \left( f(x) \right) \right) \right) = g(h(x)),
\nonumber
\end{equation}
hence $h$ provides a conjugacy as required.

Finally, if instead $a > f(b)$ we can only define $h$ up to $f^{-1}(a)$.
In this case we simply extend $h$ from $f^{-1}(a)$ to $b$ in some $C^1$ fashion
because $h$ doesn not need to satisfy the conjugacy condition for $f^{-1}(a) \le x < b$.
Similarly if $f(a) > b$ we can only define $h_1$, and hence $h$, down to $f^{-1}(b)$,
in which case we similarly extend $h$ from $f^{-1}(b)$ to $a$ in some $C^1$ fashion.
\end{proof}

\subsection{One fixed point at the point of non-differentiability}
\label{sub:mu0}

Finally we prove the existence of a differentiable conjugacy in the case that
$0$ is a fixed point of $f$ and $g$.
This requires the lower and upper derivatives of $f$ at $x=0$
are the same as the lower and upper derivatives of $g$ at $y=0$.
This is a natural generalisation of requiring equal stability multipliers in the smooth case,
and is stronger than the requirement \eqref{eq:equalSlopeRatios} of equal slope ratios.
This result is used in \S\ref{sec:matching} for the bifurcation value $\mu = 0$.

\begin{lemma}
Consider piecewise-$C^2$ maps $f$ and $g$ \eqref{eq:fg} satisfying \eqref{eq:infty}.
Suppose $f_L'(x) \ne 0$ for all $x \in (a,0]$ and $f_R'(x) \ne 0$ for all $x \in [0,b)$.
Suppose $f$ has the unique fixed point $x=0$ with $f_L'(0) = a_L \ne \pm 1$ and $f_R'(0) = a_R \ne \pm 1$.
Suppose $g_L'(y) \ne 0$ for all $y \le 0$ and $g_R'(y) \ne 0$ for all $y \ge 0$.
Suppose $g$ has the unique fixed point $y=0$ with $g_L'(0) = a_L$ and $g_R'(0) = a_R$.
In the case $a_L < 0$ and $a_R < 0$ suppose also $a_L a_R \ne 1$ and that $f^2$ and $g^2$ have no fixed points other than $0$.
Then $f$ on $(a,b)$ can be differentiably conjugated to $g$ on some subset of $\mathbb{R}$
using a conjugacy function $h$ with $h'(0) = 1$.
\label{le:muEqualsZero}
\end{lemma}

\begin{proof}
First suppose $a_L > 0$ and $a_R > 0$.
By Lemma \ref{le:globalLinearisationInc} applied to $f_L$ (smoothly extended into $x > 0$),
there exists increasing $h_1$ that differentiably conjugates $f_L$ on $(a,0]$ to
the linear map $\ell(z) = a_L z$ on $(p,0]$, for some $p < 0$.
Similarly for any $\tilde{c} < 0$ there exists increasing $h_2$
that differentiably conjugates $g_L$ on $(\tilde{c},0]$ to $\ell$ on $(r,0]$, for some $r < 0$,
Define the linear function $h_3(z) = \frac{h_2'(0) z}{h_1'(0)}$.
Then $h = h_2^{-1} \circ h_3 \circ h_1$ has $h'(0) = 1$ and differentiably conjugates $f_L$ on $(a,0]$ to $g_L$ on $(c,0]$,
where $c = h_2^{-1} \left( h_3(p) \right) < 0$ (this is well-defined because $r \to -\infty$ as $\tilde{c} \to -\infty$
due to \eqref{eq:infty}).
In the same way we can construct $h$ with $h'(0) = 1$ that differentiably conjugates $f_R$ on $[0,b)$ to $g_R$ on $[0,d)$, for some $d > 0$.
The two parts of $h$ have the same derivative at $x=0$ thus together provide a differentiable conjugacy
from $f$ on $(a,b)$ to $g$ on $(c,d)$.

Second suppose $a_L > 0$ and $a_R < 0$ (the case $a_L < 0$ and $a_R > 0$
follows from the substitution $x \mapsto -x$).
As above we can construct $h$ with $h'(0) = 1$
that conjugates $f_L$ on $(a,0]$ to $g_L$ on $(c,0]$, for some $c > 0$.
We then use
\begin{equation}
h(x) = g_R^{-1} \left( h(f_R(x)) \right),
\label{eq:muEqualsZeroProof10}
\end{equation}
to extend the domain of $h$ to $(0,b)$, assuming $f_R(b) \ge a$ otherwise just extend to $f_R^{-1}(a)$.
This provides a differentiable conjugacy on the larger domain and is $C^1$ at $x=0$
because the chain rule applied to \eqref{eq:muEqualsZeroProof10} gives
\begin{equation}
\lim_{x \to 0^+} h'(0) = \frac{f_R'(0)}{g_R'(0)} \,\lim_{x \to 0^-} h'(0),
\nonumber
\end{equation}
which equals $\lim_{x \to 0^-} h'(0) = 1$ because $f_R'(0) = g_R'(0)$ by assumption.
If $f_R(b) < a$ we then extend the domain of $h(x)$ from $f_R^{-1}(a)$ to $b$ in some $C^1$ fashion.

Finally suppose $a_L < 0$ and $a_R < 0$;
also suppose $a < f(b)$ and $f(a) < b$ (other cases can be dealt with similarly).
Notice $f^2$ and $g^2$ are increasing $C^1$ maps on $(a,b)$ and $\mathbb{R}$ respectively.
Both maps $f^2$ and $g^2$ have the unique fixed point $0$
for which the associated stability multiplier is $a_L a_R \ne 1$.
By applying the above steps to $f^2$ and $g^2$ (instead of $f_L$ and $g_L$),
we can construct a $C^1$ function $h_4$ with $h_4'(0) = 1$
that differentiably conjugates $f^2$ on $(a,0]$ to $g^2$ on $(c,0]$, for some $c < 0$, i.e.
\begin{equation}
h_4 \left( f^2(x) \right) = g^2(h_4(x)),
\nonumber
\end{equation}
for all $x \in (a,0]$.
Now define
\begin{equation}
h(x) = \begin{cases}
h_4(x), & a < x < 0, \\
g^{-1} \left( h_4 \left( f(x) \right) \right), & 0 < x < b,
\end{cases}
\nonumber
\end{equation}
which exclusively uses $f = f_R$ and $g = g_R$ in the second piece.
As in the proof of Lemma \ref{le:constructLR} this gives $h(f(x)) = g(h(x))$ for all $x \in (a,b)$.
Moreover, $\lim_{x \to 0^-} h'(0) = h_4'(0) = 1$,
and by the chain rule
\begin{equation}
\lim_{x \to 0^+} h'(0) = \frac{h_4'(0) f_R'(0)}{g_R'(0)} = h_4'(0) = 1,
\nonumber
\end{equation}
because $f_R'(0) = g_R'(0)$.
Hence $h$ is $C^1$ with $h'(0) = 1$ and differentiably conjugates $f$ on $(a,b)$ to $g$ on $(c,d)$,
where $d = g_R^{-1} \left( h_4 \left( f_R(b) \right) \right) > 0$.
\end{proof}

\section{Matching stability multipliers and derivative ratios}
\label{sec:matching}

In this section we prove Theorems \ref{th:noBif}, \ref{th:snLike}, and \ref{th:pdLike}.
In each case this is achieved by first defining the parameters of the extended normal form $g$
in such a way that the stability multipliers of fixed points and period-two solutions of $g$
match the corresponding stability multipliers of $f$, and so that $f$ and $g$
have the same ratio of lower and upper derivatives at $0$.
Second we apply the results of the previous section to show $f$ and $g$ are differentiably conjugate on certain intervals,
and finally we verify the bound \eqref{eq:theBound}.

In each case the values of $a_L$ and $a_R$ are not $1$,
so in a neighbourhood of $(x;\mu) = (0;0)$ the left and right pieces of $f$ have locally unique fixed points
\begin{align}
x_L(\mu) &= \frac{\beta}{1-a_L} \,\mu + \cO \left( \mu^2 \right), \label{eq:xL} \\
x_R(\mu) &= \frac{\beta}{1-a_R} \,\mu + \cO \left( \mu^2 \right). \label{eq:xR}
\end{align}
The point $x_L(\mu)$ is a fixed point of $f$ if $x_L(\mu) \le 0$
(if instead $x_L(\mu) > 0$ it is a {\em virtual} fixed point of $f$ \cite{DiBu08}),
while $x_R(\mu)$ is a fixed point of $f$ if $x_R(\mu) \ge 0$.
The stability multipliers of the fixed points are
\begin{align}
\lambda_L(\mu) &= \frac{\partial f_L}{\partial x}(x_L(\mu);\mu) = a_L + \cO(\mu), \label{eq:stabMultxL} \\
\lambda_R(\mu) &= \frac{\partial f_R}{\partial x}(x_R(\mu);\mu) = a_R + \cO(\mu). \label{eq:stabMultxR}
\end{align}

The extended normal form $g$ has three parameters, $s_L$, $s_R$, and $t$,
and in each of the proofs below we set the first two of these parameters as
\begin{align}
s_L(\mu) &= \begin{cases}
\lambda_L(\mu), & \mu \le 0, \\
\frac{\frac{\partial f_L}{\partial x}(0;\mu) \lambda_R(\mu)}{\frac{\partial f_R}{\partial x}(0;\mu)}, & \mu \ge 0,
\end{cases} \label{eq:sLproof1} \\
s_R(\mu) &= \begin{cases}
\frac{\frac{\partial f_R}{\partial x}(0;\mu) \lambda_L(\mu)}{\frac{\partial f_L}{\partial x}(0;\mu)}, & \mu \le 0, \\
\lambda_R(\mu), & \mu \ge 0.
\end{cases} \label{eq:sRproof1}
\end{align}
This is because every scenario we examine in detail below has $a_R < 1$,
so with $\beta > 0$ the point $x_R(\mu)$ is a fixed point of $f$ for small $\mu > 0$.
The corresponding fixed point of $g$ has stability multiplier $s_R(\mu)$, because the right piece of $g$ is affine,
hence for $\mu > 0$ we must set $s_R(\mu) = \lambda_R(\mu)$
so that these two fixed points share the same stability multiplier.
For $s_L(\mu)$ the $\mu > 0$ piece of \eqref{eq:sLproof1} ensures
$f$ and $g$ have the same ratio of lower and upper derivatives at $0$.
For $\mu < 0$ the formulas \eqref{eq:sLproof1} and \eqref{eq:sRproof1}
ensure the left fixed points of $f$ and $g$ share the same stability multiplier
in the proofs of Theorems \ref{th:noBif} and \ref{th:pdLike};
for the proof of Theorem \ref{th:snLike} we have more flexibility
but can again use \eqref{eq:sLproof1} and \eqref{eq:sRproof1}.
For $t(\mu)$ we use a different function in each of the three proofs.

\begin{proof}[Proof of Theorem \ref{th:noBif}]
We first address the increasing case, and for brevity suppose $a_L, a_R \in (0,1)$
because the case $a_L, a_R \in (1,\infty)$ can be proved in the same way using the inverses of $f$ and $g$.

\myStep{1}{Match $g$ to $f$.}
Let $\tilde{g}(y;\mu) = g(y;\nu(\mu),s_L(\mu),s_R(\mu),t(\mu))$
where $\nu(\mu) = f(0;\mu) = \beta \mu + \cO \left( \mu^2 \right)$ (as in the theorem statement),
$s_L(\mu)$ and $s_R(\mu)$ are defined by \eqref{eq:sLproof1} and \eqref{eq:sRproof1}, and $t(\mu) = 0$.
Notice $s_L(\mu)$ and $s_R(\mu)$ are continuous functions of $\mu$ with $s_L(0) = a_L$ and $s_R(0) = a_R$,
and it what follows we assume $\mu$ is small enough that $s_L(\mu), s_R(\mu) \in (0,1)$.
The left and right pieces of $\tilde{g}$ are
\begin{align}
g_L(y;\mu) &= \nu(\mu) + s_L(\mu) y, \\
g_R(y;\mu) &= \nu(\mu) + s_R(\mu) y.
\end{align}
These have unique fixed points
\begin{align}
y_L(\mu) &= \frac{\nu(\mu)}{1 - s_L(\mu)} = \frac{\beta}{1 - a_L} \,\mu + \co(\mu), \\
y_R(\mu) &= \frac{\nu(\mu)}{1 - s_R(\mu)} = \frac{\beta}{1 - a_R} \,\mu + \co(\mu),
\end{align}
being $C^1$ functions of $\mu$.

\myStep{2}{Obtain a differentiable conjugacy $h$.}
Now fix $\mu_0, p > 0$ and consider $\mu \in (-\mu_0,0)$.
We assume $\mu_0, p > 0$ have been taken small enough that
$f_L'(x;\mu) > 0$ for all $x \in (-p,0]$,
$f_R'(x;\mu) > 0$ for all $x \in [0,p)$,
and $x_L(\mu) \in (-p,p)$ is the unique fixed point of $f$ on $(-p,p)$ as in Fig.~\ref{fig:pwsDiffConjMain}.
Since $s_L(\mu), s_R(\mu) \in (0,1)$ and $\beta > 0$,
the point $y_L(\mu)$ is the unique fixed point of $\tilde{g}$ on $\mathbb{R}$.
Its stability multiplier is $s_L(\mu)$, which is the same as the stability multiplier of $x_L(\mu)$;
also $f$ and $\tilde{g}$ have the same ratio of lower and upper derivatives at $0$ by our above definition of $s_L(\mu)$ and $s_R(\mu)$.
Thus by Lemma \ref{le:constructOneInc} applied to $f$ on $(-p,p)$ and $\tilde{g}$ on $\mathbb{R}$,
for all $\mu \in (-\mu_0,0)$ there exists $c(\mu) < 0$, $d(\mu) > 0$, and a diffeomorphism
$h : (-p,p) \to (c(\mu),d(\mu))$ that conjugates $f$ to $\tilde{g}$.
For small $\mu > 0$ such a diffeomorphism can be constructed
in a similar way based on the equality of the stability multipliers of the fixed points of the right pieces of $f$ and $\tilde{g}$.
For $\mu = 0$ such a diffeomorphism $h$ exists by Lemma \ref{le:muEqualsZero} having also $h'(0) = 1$.

\begin{figure}[h!]
\begin{center}
\includegraphics[height=9cm]{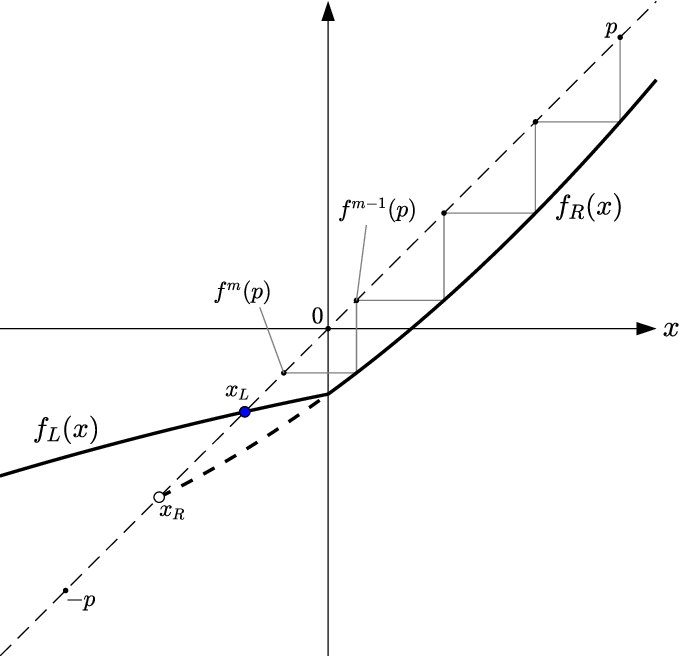} 
\caption{
A sketch of the map $f$ with $\mu < 0$ to support the proof of Theorem \ref{th:noBif}.
The fixed point $x_L(\mu)$ of $f_L$ is a fixed point of $f$,
while the fixed point $x_R(\mu)$ of $f_R$ is a virtual fixed point of $f$.
\label{fig:pwsDiffConjMain}
} 
\end{center}
\end{figure}

\myStep{3}{Bound the left image point $c(\mu)$.}
To complete the proof in the increasing case
it remains for us to verify \eqref{eq:theBound}
using $N = (-p,p)$ and $M(\mu) = (c(\mu),d(\mu))$.
That is, we need to show that the ratios
$\frac{|c(\mu)|}{p}$ and $\frac{d(\mu)}{p}$
can be made as close to $1$ as we like by choosing $\mu_0, p > 0$ suitably small.
For brevity we demonstrate this just for $\mu \in (-\mu_0,0)$
as the result for $\mu \in [0,\mu_0)$ can be proved in the same way
(using the fact that $h'(0) = 1$ in the case $\mu = 0$).

We first apply Lemma \ref{le:bound} to $f_L$ on $(a,b) = (-p,0)$ and $g_L$ on $(-\infty,0)$
having fixed points $x^* = x_L(\mu)$ and $y^* = y_L(\mu)$, respectively.
Here
\begin{equation}
\chi = \frac{y_L(\mu)}{x_L(\mu)} = 1 + \cO(\mu),
\label{eq:chi1}
\end{equation}
and by \eqref{eq:bound} with $x = -p$ there exists $r > 0$ such that
\begin{equation}
\left| \frac{c(\mu) - y_L(\mu)}{-p - x_L(\mu)} - \chi \right| \le \frac{3}{2 r} \big( 4 |p+x_L(\mu)| + |x_L(\mu)| \big),
\label{eq:boundApplied_a1}
\end{equation}
for all sufficiently small $\mu > 0$.
The precise value on the right-hand side of \eqref{eq:boundApplied_a1} is unimportant,
it is sufficient for us to observe this value is $\cO(\mu_0,p)$.
By multiplying both sides of \eqref{eq:boundApplied_a1} by $1 + \frac{x_L(\mu)}{p} \in (1,2)$ we obtain
\begin{equation}
\left| \frac{|c(\mu)|}{p} - \chi \right| = \cO(\mu_0,p).
\nonumber
\end{equation}
Thus $\frac{|c(\mu)|}{p} \to \chi \to 1$ as $\mu$ and $p$ tend to $0$,
so indeed for any $\delta > 0$ there exists $\mu_0, p > 0$
such that $1 - \delta < \frac{|c(\mu)|}{p} < 1 + \delta$
for all $\mu \in (-\mu_0,0)$.

\myStep{4}{Bound the left image point $d(\mu)$.}
To prove the same result for $\frac{d(\mu)}{p}$,
let $m = m(\mu) \ge 1$ be such that $f^m(p;\mu) < 0$ and $f^{m-1}(p;\mu) \ge 0$, see Fig.~\ref{fig:pwsDiffConjMain}.
We again apply Lemma \ref{le:bound} to $f_L$ and $g_L$, but now in \eqref{eq:bound} use $x = f^m(p;\mu)$ to obtain
\begin{equation}
\left| \frac{h \left( f^m(p;\mu) \right) - y_L(\mu)}{f^m(p;\mu) - x_L(\mu)} - \chi \right|
\le \frac{3}{2 r} \big( 4 |f^m(p;\mu) + x_L(\mu)| + |x_L(\mu)| \big).
\nonumber
\end{equation}
Observe $f^m(p;\mu) = f_R^m(p;\mu)$
and $h \left( f^m(p;\mu) \right) = \tilde{g}^m(d(\mu);\mu) = g_R^m(d(\mu);\mu)$, so
\begin{equation}
\left| \frac{g_R^m(d(\mu);\mu) - y_L(\mu)}{f_R^m(p;\mu) - x_L(\mu)} - \chi \right| = \cO(\mu_0,p).
\label{eq:boundApplied_b1}
\end{equation}
It is a simple exercise to show there exists $K \ge 0$ such that
\begin{equation}
f^m(p;\mu) - x_R(\mu) \le K \lambda_R^m p,
\label{eq:K}
\end{equation}
for all sufficiently small $\mu, p > 0$ (this can be achieved by performing
calculations similar to those in the proof of Lemma \ref{le:fn}a).
To approximate $p$ from $f^m(p;\mu)$ we apply Lemma \ref{le:fn}b to $f_R$,
shifted so that its fixed point $x_R(\mu)$ is $0$.
This fixed point has stability multiplier $\lambda_R(\mu)$,
so using $n = m$ and $v = f^m(p;\mu) - x_R(\mu)$ in \eqref{eq:fInvnBound} gives
\begin{equation}
\left| p - x_R(\mu) - \frac{f^m(p;\mu) - x_R(\mu)}{\lambda_R(\mu)^m} \right|
\le \frac{2 \left( f^m(p;\mu) - x_R(\mu) \right)^2}{5 r \lambda_R(\mu)^m}
\le \frac{2 K^2 p}{5},
\label{eq:boundApplied_b2}
\end{equation}
using also \eqref{eq:K} and assuming $p \le r$.
To similarly express $d(\mu)$ in terms of $g_R^m(d(\mu);\mu)$,
observe $g_R$ is affine with slope $s_R(\mu)$ and fixed point $y_R(\mu)$, so we have the exact formula
\begin{equation}
d(\mu) - y_R(\mu) = \frac{g_R^m(d(\mu);\mu) - y_R(\mu)}{s_R(\mu)^m}.
\label{eq:boundApplied_b3}
\end{equation}
From \eqref{eq:boundApplied_b1}, \eqref{eq:boundApplied_b2}, and \eqref{eq:boundApplied_b3},
direct calculations show $\frac{d(\mu)}{p} \to 1$ as $\mu$ and $p$ tend to $0$.

\myStep{5}{Repeat for the remaining cases.}
Now suppose $a_L, a_R \in (-1,0)$ (the case $a_L, a_R \in (-\infty,-1)$ can be handled similarly)
and define $s_L(\mu)$, $s_R(\mu)$, and $t(\mu)$ as above.
Fix $\mu_0, p > 0$ and choose any $\mu \in (-\mu_0,0)$.
We can assume $\mu_0, p > 0$ have been taken small enough that
Lemma \ref{le:constructOneDec} can be applied to $f$ on $(-p,p)$ and $\tilde{g}$ on $\mathbb{R}$.
Thus there exists $c(\mu) < 0$, $d(\mu) > 0$, and a diffeomorphism
$h : (-p,p) \to (c(\mu),d(\mu))$ that conjugates $f$ to $\tilde{g}$.
For small $\mu > 0$ a differentiable conjugacy exists for the same reasons;
for $\mu = 0$ a differentiable conjugacy exists by Lemma \ref{le:muEqualsZero}.

If instead $a_L \in (0,1)$ and $a_R \in (-1,0)$, or $a_L \in (-1,0)$ and $a_R \in (0,1)$,
the existence of a differentiable conjugacy $h$
follows from Lemma \ref{le:constructOneNonMonInc} for one sign of $\mu$,
Lemma \ref{le:constructOneNonMonDec} for the other sign of $\mu$,
and Lemma \ref{le:muEqualsZero} for $\mu = 0$.
In each case \eqref{eq:theBound} can be verified through calculations similar to those given above.
\end{proof}

\begin{proof}[Proof of Theorem \ref{th:snLike}]
We first treat the monotone case $a_R \in (0,1)$ with $\mu > 0$.

\myStep{1}{Match $g$ to $f$.}
With $\mu > 0$ the map $f$ has two fixed points $x_L(\mu)$ and $x_R(\mu)$, given by \eqref{eq:xL} and \eqref{eq:xR},
with stability multipliers $\lambda_L(\mu)$ and $\lambda_R(\mu)$, given by \eqref{eq:stabMultxL} and \eqref{eq:stabMultxR}.
Since $f$ is piecewise-$C^3$ the stability multipliers are $C^2$ functions of $\mu$.

We define $s_L(\mu)$ and $s_R(\mu)$ by \eqref{eq:sLproof1} and \eqref{eq:sRproof1}
so that the right fixed point of $g$ has the same stability multiplier as $x_R(\mu)$
and $f$ and $g$ have the same ratio of lower and upper derivatives at $0$.
We now define $t(\mu)$ so that the left fixed point of $g$ has the same stability multiplier as $x_L(\mu)$.
This requires second-order asymptotics, so let us write
\begin{equation}
\begin{split}
f_L(x;\mu) &= a_L x + \beta \mu + c_L x^2 + d_L \mu x + e \mu^2 + \cO \left( \left( |x| + |\mu| \right)^3 \right), \\
f_R(x;\mu) &= a_R x + \beta \mu + c_R x^2 + d_R \mu x + e \mu^2 + \cO \left( \left( |x| + |\mu| \right)^3 \right),
\end{split}
\label{eq:fLfRtoSecondOrder}
\end{equation}
where $c_L, c_R, d_L, d_R, e \in \mathbb{R}$ are constants.
Using \eqref{eq:sLproof1} and \eqref{eq:sRproof1} we obtain
\begin{align}
s_L(\mu) &= a_L + \left( d_L + \frac{2 a_L \beta c_R}{a_R (1-a_R)} \right) \mu + \cO \left( \mu^2 \right), \label{eq:sL2} \\
s_R(\mu) &= a_R + \left( d_R + \frac{2 \beta c_R}{1 - a_R} \right) \mu + \cO \left( \mu^2 \right);
\end{align}
also
\begin{equation}
\lambda_L(\mu) = a_L + \left( d_L + \frac{2 \beta c_L}{1 - a_L} \right) \mu + \cO \left( \mu^2 \right).
\end{equation}
Let $\nu(\mu) = f(0;\mu)$ and let
\begin{equation}
\hat{g}_L(y;\mu,t) = \nu(\mu) + s_L(\mu) y + t y^2,
\nonumber
\end{equation}
be the left piece of $g$.
The map $\hat{g}_L$ has a fixed point
$\hat{y}_L(\mu,t) = \frac{\beta}{1 - a_L} \,\mu + \cO \left( \mu^2 \right)$
with stability multiplier
\begin{equation}
\hat{\lambda}_L(\mu,t) = \frac{\partial \hat{g}_L}{\partial y} \big( \hat{y}_L(\mu,t); \mu, t \big)
= a_L + \left( d_L + \frac{2 a_L \beta c_R}{a_R(1-a_R)} + \frac{2 t}{1 - a_L} \right) \mu + \cO \left( \mu^2 \right),
\end{equation}
using \eqref{eq:sL2}.
Now let
\begin{equation}
U(\mu,t) = \begin{cases}
\frac{1}{\mu} \left( \hat{\lambda}_L(\mu,t) - \lambda_L(\mu) \right), & \mu > 0, \\
\frac{\partial}{\partial \mu} \left( \hat{\lambda}_L(\mu,t) - \lambda_L(\mu) \right), & \mu = 0.
\end{cases}
\label{eq:Ksn}
\end{equation}
Notice $U$ is $C^1$ and
\begin{equation}
U(\mu,t) = \frac{2 (t - \beta c_L)}{1 - a_L} + \frac{2 a_L \beta c_R}{a_R(1-a_R)} + \cO(\mu).
\nonumber
\end{equation}
Thus by the implicit function theorem there exists a unique $C^1$ function
\begin{equation}
t(\mu) = \beta c_L - \frac{a_L(1-a_L) \beta c_R}{a_R (1 - a_R)} + \cO(\mu),
\label{eq:tsn}
\end{equation}
so that $U(\mu,t(\mu)) = 0$ for all $\mu$ in a neighbourhood of $0$.

\myStep{2}{Obtain two differentiable conjugacies for $\mu > 0$.}
We now show $f$ is differentiably conjugate to
$\tilde{g}(y;\mu) = g(y;\nu(\mu),s_L(\mu),s_R(\mu),t(\mu))$,
where $\nu(\mu) = f(0;\mu)$
and $s_L(\mu)$, $s_R(\mu)$, and $t(\mu)$ are given by \eqref{eq:sLproof1}, \eqref{eq:sRproof1}, and \eqref{eq:tsn}.
The left and right pieces of $\tilde{g}$ are
\begin{align}
g_L(y;\mu) &= \nu(\mu) + s_L(\mu) y + t(\mu) y^2, \\
g_R(y;\mu) &= \nu(\mu) + s_R(\mu) y,
\end{align}
which have locally unique fixed points
\begin{align}
y_L(\mu) &= \frac{\nu(\mu)}{1 - s_L(\mu)} = \frac{\beta}{1 - a_L} \,\mu + \cO \left( \mu^2 \right), \\
y_R(\mu) &= \frac{\nu(\mu)}{1 - s_R(\mu)} = \frac{\beta}{1 - a_R} \,\mu + \cO \left( \mu^2 \right).
\end{align}
Notice the stability multiplier of $y_L(\mu)$, namely $\lambda_L(\mu)$,
is the same as the stability multiplier of $x_L(\mu)$
in view of the way we have constructed $t(\mu)$.

Fix $\mu_0, p > 0$ and consider $\mu \in (0,\mu_0)$.
We can assume $\mu_0, p > 0$ are small enough that $f$
satisfies the conditions in Lemma \ref{le:constructTwoMon} on the interval $(a,b) = (-p,p)$.
However, $\tilde{g}$ may not satisfy the conditions on $g$
in Lemma \ref{le:constructTwoMon} far from $y=0$ due to the $y^2$-term in $g_L$,
so we work instead with a piecewise-$C^2$ map $\tilde{g}^*$ on $\mathbb{R}$
that is identical to $\tilde{g}$ for all $y \ge -2 p$
and is defined for $y < -2 p$ in such a way that it does satisfy the conditions in Lemma \ref{le:constructTwoMon}.
Then, by Lemma \ref{le:constructTwoMon},
for all $\mu \in (0,\mu_0)$ there exist $c(\mu) < 0$, $d(\mu) > 0$, and diffeomorphisms
$h_1 : (-p,x_R(\mu)) \to (c(\mu),y_R(\mu))$ and
$h_2 : (x_L(\mu),p) \to (y_L(\mu),d(\mu))$ that conjugate $f$ to $\tilde{g}^*$.
As in the proof of Theorem \ref{th:noBif}
we can use Lemma \ref{le:bound} to show
$\frac{|c(\mu)|}{p} \to 1$ and $\frac{d(\mu)}{p} \to 1$ as $\mu \to 0$ verifying \eqref{eq:theBound}.
In particular, we can get $c(\mu) > -2 p$
so $h_1$ and $h_2$ differentiably conjugate $f$ to $\tilde{g}$ as required.

\myStep{3}{Obtain a differentiable conjugacy for $\mu \le 0$.}
For $\mu = 0$ the existence of a differentiable conjugacy
follows from Lemma \ref{le:muEqualsZero},
while for $\mu < 0$ the map $f$ has no fixed points
so we use Lemma \ref{le:constructZeroMon}.
Notice we can assume $\mu_0, p > 0$ are small enough that for all $\mu \in (-\mu_0,0)$ the map $f$
satisfies the conditions in Lemma \ref{le:constructZeroMon} on the interval $(a,b) = (-p,p)$.
If necessary we modify $\tilde{g}$ for $y < -2 p$ to form $\tilde{g}^*$ satisfying
the conditions on $g$ in Lemma \ref{le:constructZeroMon} on $\mathbb{R}$.
Then by Lemma \ref{le:constructZeroMon} there exist $c(\mu) < 0$, $d(\mu) > 0$,
and a diffeomorphism $h : (-p:p) \to (c(\mu),d(\mu))$ that conjugates $f$ to $\tilde{g}^*$.

To verify \eqref{eq:theBound},
in the construction used for the proof of Lemma \ref{le:constructZeroMon}
we can take $z = {\rm min} \left( f^{-m}(-p), f^n(p) \right)$,
where $m \ge 1$ is such that $f^{-m}(-p) < 0 < f^{-(m+1)}(-p)$
and $n \ge 1$ is such that $f^n(p) < 0 < f^{n-1}(p)$
(the special case that $0$ belongs to the backward orbit of $-p$ or the forward orbit of $p$
can be accommodated by slightly increasing the value of $p$).
Then, by \eqref{eq:constructZeroMonProof1},
we have $g^{-m}(c(\mu)) = h \left( f^{-m}(-p) \right) = f^{-m}(-p)$
and $g^n(d(\mu)) = h(f^n(p)) = f^n(p)$.
Via the approach used in Step 4 of the previous proof,
we get $\frac{|c(\mu)|}{p}$ and $\frac{d(\mu)}{p}$ converging to $1$ as $\mu_0, p \to 0$ verifying \eqref{eq:theBound}.
So again we can assume $c(\mu) > -2 p$ and hence $f$ on $(-p,p)$ is differentiably conjugate to $\tilde{g}$ on $(c(\mu),d(\mu))$.

\myStep{4}{Repeat for the non-monotone case.}
If instead $a_R \in (-1,0)$ we define $s_L(\mu)$, $s_R(\mu)$, and $t(\mu)$ as above.
The result follows from Lemma \ref{le:constructTwoNonMon} for $\mu > 0$,
Lemma \ref{le:muEqualsZero} for $\mu = 0$,
and Lemma \ref{le:constructZeroNonMon} for $\mu < 0$.
In each case \eqref{eq:theBound} can be verified through calculations similar to those given in the previous proof.
\end{proof}

\begin{proof}[Proof of Theorem \ref{th:pdLike}]~

\myStep{1}{Match $g$ to $f$.}
First consider $\mu > 0$.
In this case $f$ has a locally unique fixed point $x_R(\mu)$, given by \eqref{eq:xR},
with stability multiplier $\lambda_R(\mu)$, given by \eqref{eq:stabMultxR},
and a locally unique period-two solution
consisting of points $u_L(\mu) < 0$ and $u_R(\mu) > 0$.
By direct calculations we find that the stability multiplier of the period-two solution is
\begin{equation}
\xi(\mu) = a_L a_R + \left( a_L d_R + a_R d_L + \frac{2 (a_L(1+a_L)c_R + a_R(1+a_R)c_L) \beta}{1 - a_L a_R} \right) \mu
+ \cO \left( \mu^2 \right),
\label{eq:fLRcycleEig}
\end{equation}
where again we write $f_L$ and $f_R$ as \eqref{eq:fLfRtoSecondOrder}.

Define $s_L(\mu)$ and $s_R(\mu)$ by \eqref{eq:sLproof1} and \eqref{eq:sRproof1}
so that the right fixed point of $g$ has the same stability multiplier as $x_R(\mu)$
and $f$ and $g$ have the same ratio of lower and upper derivatives at $0$.
We now work to defining $t(\mu)$ such that $g$ has a period-two solution with stability multiplier $\xi(\mu)$.
The map $\hat{g}(y;\mu,t) = g(y;\nu(\mu),s_L(\mu),s_R(\mu),t)$ has a period-two solution with stability multiplier
\begin{equation}
\hat{\xi}(\mu,t) = a_L a_R + \left( a_L d_R + a_R d_L
+ \frac{4 a_L \beta c_R}{1 - a_R} + \frac{2 a_R (1+a_R) \beta t}{1 - a_L a_R} \right) \mu
+ \cO \left( \mu^2 \right).
\label{eq:gLRcycleEig}
\end{equation}
Notice $\xi(\mu)$ and $\hat{\xi}(\mu,t)$ are $C^2$ because $f$ is piecewise-$C^3$.
So the function
\begin{equation}
U(\mu,t) = \begin{cases}
\frac{1}{\mu} \left( \hat{\xi}(\mu,t) - \xi(\mu) \right), & \mu > 0, \\
\frac{\partial}{\partial \mu} \left( \hat{\xi}(\mu,t) - \xi(\mu) \right), & \mu = 0,
\end{cases}
\label{eq:Kpd}
\end{equation}
is $C^1$; thus by the implicit function theorem there exists a unique $C^1$ function
\begin{equation}
t(\mu) = \beta c_L + \frac{a_L(1+a_L)c_R}{a_R(1+a_R)} - \frac{2(1-a_L a_R) a_L \beta c_R}{a_R(1+a_R)(1-a_R)} + \cO(\mu),
\label{eq:tpd}
\end{equation}
giving $U(\mu,t(\mu)) = 0$ for all $\mu$ in a neighbourhood of $0$.

\myStep{2}{Obtain two differentiable conjugacies for $\mu > 0$.}
We now use Lemma \ref{le:constructLR} to show $f$ is differentiably conjugate to
$\tilde{g}(y;\mu) = g(y;\nu(\mu),s_L(\mu),s_R(\mu),t(\mu))$.
The map $\tilde{g}$ has a locally unique fixed point $y_R(\mu) > 0$ and a locally unique period-two solution comprised of points
$v_L(\mu) < 0$ and $v_R(\mu) > 0$.
By construction the stability multipliers of these
are identical to those of $x_R(\mu)$ and $\{ u_L(\mu), u_R(\mu) \}$.

Fix $\mu_0, p > 0$ and consider $\mu \in (0,\mu_0)$.
We can assume $\mu_0, p > 0$ are small enough that $f$
satisfies the conditions in Lemma \ref{le:constructLR} on $(a,b) = (-p,p)$,
and $\tilde{g}^*$ satisfies the conditions on $g$ in Lemma \ref{le:constructLR} on $\mathbb{R}$,
where $\tilde{g}^*$ is piecewise-$C^2$ and equal to $\tilde{g}$ for all $y \ge -2 p$.
So, by Lemma \ref{le:constructLR},
for all $\mu \in (0,\mu_0)$ there exist $c(\mu) < 0$, $d(\mu) > 0$, and diffeomorphisms
$h_1 : (u_L(\mu),u_R(\mu)) \to (v_L(\mu),v_R(\mu))$ and
$h_2 : (-p,x_R(\mu)) \cup (x_R(\mu),p) \to (c(\mu),y_R(\mu)) \cup (y_R(\mu),d(\mu))$ that conjugate $f$ to $\tilde{g}^*$.
The bound \eqref{eq:theBound}
can be verified as in the proof of Theorem \ref{th:noBif}
but by working with $f^2$ instead of $f$.
In particular this gives $c(\mu) > -2 p$
so $h_1$ and $h_2$ differentiably conjugate $f$ to $\tilde{g}$ as required.

\myStep{3}{Obtain a differentiable conjugacy for $\mu \le 0$.}
For $\mu = 0$ the result follows from Lemma \ref{le:muEqualsZero},
while for $\mu < 0$ the result follows from Lemma \ref{le:constructOneDec}
as in the decreasing case of Theorem \ref{th:noBif}.
Again in each case \eqref{eq:theBound} can be verified in a similar fashion.
\end{proof}

\section{Remarks on multi-dimensional maps}
\label{sec:nd}

Mathematical models are rarely one-dimensional so it is important to understand
how bifurcations behave in multi-dimensional maps.
For classical (smooth) bifurcations this is usually
achieved by using a centre manifold to reduce the dimension of the problem,
e.g.~to one dimension for saddle-node bifurcations.
Such reduction is usually not possible for border-collision bifurcations
because there are no eigenvalues with unit modulus that are needed to form a centre manifold.

For a border-collision bifurcation of a map $f$ on $\mathbb{R}^n$,
one can replace each piece of the map with an affine map,
then to the resulting piecewise-linear approximation perform a change of coordinates
to obtain an instance of the $n$-dimensional border-collision normal form \cite{Si16}.
In two dimensions the normal form is
\begin{equation}
\begin{bmatrix} y_1 \\ y_2 \end{bmatrix} \mapsto
\begin{cases}
\begin{bmatrix} \tau_L y_1 + y_2 + \nu \\ -\delta_L y_1 \end{bmatrix}, & y_1 \le 0, \\[4mm]
\begin{bmatrix} \tau_R y_1 + y_2 + \nu \\ -\delta_R y_1 \end{bmatrix}, & y_1 \ge 0,
\end{cases}
\label{eq:2dbcnf}
\end{equation}
where $\nu \in \mathbb{R}$ is the primary bifurcation parameter
and $\tau_L, \delta_L, \tau_R, \delta_R \in \mathbb{R}$ are additional parameters,
akin to the slopes $s_L$ and $s_R$ in the one-dimensional version \eqref{eq:pwl}.
There have been many studies that describe the dynamics of \eqref{eq:2dbcnf}
\cite{BaGr99,FaSi23,GhMc24,NuYo92,SuGa08,ZhMo06b}, 
but do these dynamics actually occur in the original map $f$?

Broadly speaking, yes, in a neighbourhood of the bifurcation,
but there are exceptions and an general absence of formal results to clarify matters.
In some situations it can be proved that a particular invariant set of \eqref{eq:2dbcnf} must exist also for $f$.
For example if \eqref{eq:2dbcnf} with $\nu > 0$ has a hyperbolic period-$p$ solution with no points on the switching manifold,
then $f$ has a period-$p$ solution on one side of the border-collision bifurcation;
this is a simple consequence of the implicit function theorem \cite{SiMe10}.
If \eqref{eq:2dbcnf} has a trapping region for $\nu > 0$,
then $f$ has a topological attractor on one side of the border-collision bifurcation.
If one can also identify an invariant expanding cone for the evolution of tangent vectors in tangent space,
then this attractor must be chaotic \cite{MiSt18,SiGl24}.

But there do not appear to be any formal results explaining 
when {\em all} of the dynamics of \eqref{eq:2dbcnf} occurs also for $f$,
which is what one gets with a conjugacy.
Such a result would surely be extremely valuable.
To construct a differentiable conjugacy
one could start with a local conjugacy about a fixed point using a
higher dimensional version of Sternberg's linearisation theorem \cite{St58},
then grow the domain of the conjugacy outwards, but it is not immediately clear how to
grow the domain through the switching manifold
and whether or not reasonable constraints exist that would make this possible.

We should also bear in mind that from an applied perspective the definition
of a conjugacy may not be appropriate for piecewise-smooth systems.
As explained by di Bernardo {\em et al.}~\cite{DiBu08},
the different pieces of a piecewise-smooth map usually have distinct physical interpretations,
thus the conjugacy should preserve the switching manifold if these interpretations are to be maintained.
This is not a requirement in Definition \ref{df:conj},
nevertheless all the differentiable conjugacies we have constructed above do preserve the switching point,
i.e.~$h(0) = 0$.

\section{Discussion}
\label{sec:conc}

Presently there is no strict normal form theory for border-collision bifurcations.
This paper provides a first step to such a theory
by obtaining differentiable conjugacies in the one-dimensional case.
A generalisation to higher dimensions will likely be difficult
because unlike classical (smooth) bifurcations the dynamical complexity increases with dimension \cite{GlJe15}.
Useful progress may be possible in two dimensions
when one piece of the map has a one-dimensional range.
Such border-collision bifurcations are common in applications \cite{Ko05,RoHi19b,SzOs09}
where the degenerate range is due for example to
{\em sliding motion} whereby trajectories become constrained to a codimension-one surface \cite{DiKo03}.

Our results have shown that to obtain a differentiable conjugacy
the standard piecewise-linear approximation is often insufficient,
but can often be remedied by adding a quadratic term to produce an extended normal form.
The coefficients in the extended normal form
are functions of the parameter $\mu$ that controls the border-collision bifurcation.
This is another departure from the standard approach in which the coefficients in the
piecewise-linear approximation are treated as constants.

If bifurcation theory is based on topological conjugacy,
the standard approach whereby the slopes of the normal form
are matched to those of the original map is perfunctory.
Any choice of slopes in appropriate intervals will suffice.
But by choosing the slopes in the standard way,
the normal form is a truncation of the differentiably conjugate extended normal form
explaining why this choice is privileged.

It remains to study the smoothness of the conjugating function $h$ with respect to $\mu$.
Sell \cite{Se85} has shown that for Sternberg linearisation $h$
can be made to vary smoothly with respect to parameters,
and we expect it is possible to obtain a $C^1$ dependence on $\mu$ for our results.
It also remains to accommodate border-collision bifurcations that generate chaotic dynamics.
Here differentiable conjugacies are presumably too restrictive to produce useful results
because it would be necessary to match the stability multipliers of all periodic solutions,
so topological conjugacy is likely to be the most appropriate form of equivalence.

\section*{Acknowledgements}

This work was supported by Marsden Fund contract MAU2209 managed by Royal Society Te Ap\={a}rangi.

\appendix

\section{Proof of Lemma \ref{le:bound}}
\label{app:proof}
\setcounter{equation}{0}

\begin{lemma}
Let $f : (a,b) \to \mathbb{R}$ be $C^2$ with $f'(x) > 0$ for all $x \in (a,b)$.
Suppose $0 \in (a,b)$ is the unique fixed point of $f$ with $f'(0) = \lambda \in (0,1)$.
Let $K > 0$ be such that $|f''(x)| \le K$ for all $x \in (a,b)$.
Let $r = \frac{\lambda(1-\lambda)}{10 K}$ and suppose ${\rm max}(|a|,b) \le r$.
\begin{enumerate}
\item[a)]
For any $u \in (a,b)$ and $n \ge 0$,
\begin{equation}
\left| f^n(u) - \lambda^n u \right| \le \frac{\left( 1 - \lambda^n \right) \lambda^n}{5 r} \,u^2.
\label{eq:fnBound}
\end{equation}
\item[b)]
For any $v \in (a,b)$ and $n \ge 0$,
if $f^{-n}(v) \in (a,b)$ then 
\begin{equation}
\left| f^{-n}(v) - \frac{v}{\lambda^n} \right| \le \frac{2 \left( 1 - \lambda^n \right)}{5 \lambda^{2 n} r} \,v^2.
\label{eq:fInvnBound}
\end{equation}
\end{enumerate}
\label{le:fn}
\end{lemma}

\begin{proof}
The bound $|f''(x)| \le K$ implies
\begin{equation}
|f(x) - \lambda x| \le \frac{K x^2}{2},
\label{eq:fnProof-fBound}
\end{equation}
for any $x \in (a,b)$.
We now apply the same calculation to $f^{-1}$.
We have $\left| \left( f^{-1} \right)''(y) \right|
= \left| \frac{f''(x)}{\left( f'(x) \right)^3} \right|
\le \frac{K}{\left| f'(x) \right|^3}$, where $y = f(x)$.
But $\left| f'(x) \right| \ge \lambda - K |x| \ge \lambda - K r > \frac{9 \lambda}{10} > 2^{-\frac{1}{3}} \lambda$,
so $\left| \left( f^{-1} \right)''(y) \right| \le \frac{2 K}{\lambda^3}$, and hence
\begin{equation}
\left| f^{-1}(y) - \frac{y}{\lambda} \right| \le \frac{K y^2}{\lambda^3},
\label{eq:fnProof-fInvBound}
\end{equation}
for any $y \in \mathbb{R}$ for which $f^{-1}(y) \in (a,b)$.

We now verify \eqref{eq:fnBound} by induction on $n$.
Equation \eqref{eq:fnBound} holds with $n = 0$ because in this case both sides are zero.
Suppose \eqref{eq:fnBound} holds for some $n \ge 0$.
This implies
\begin{equation}
\left| f^n(u) \right| \le \lambda^n |u| + \frac{\lambda^n u^2}{5 r} \le \frac{6 \lambda^n |u|}{5} \le 2 \lambda^n |u|,
\label{eq:fnProof-10}
\end{equation}
using $|u| \le r$ for the middle inequality.
Then by \eqref{eq:fnProof-fBound}, \eqref{eq:fnProof-10}, and the induction hypothesis,
\begin{align}
\left| f^{n+1}(u) - \lambda^{n+1} u \right|
&\le \left| f \left( f^n(u) \right) - \lambda f^n(u) \right| + \lambda \left| f^n(u) - \lambda^n u \right| \nonumber \\
&\le \frac{K}{2} \big( 2 \lambda^n |u| \big)^2 + \lambda \,\frac{\left( 1 - \lambda^n \right) \lambda^n}{5 r} \,u^2 \nonumber \\
&= \frac{\left( 1 - \lambda^{n+1} \right) \lambda^{n+1}}{5 r} \,u^2, \nonumber
\end{align}
where we have substituted $K = \frac{\lambda(1-\lambda)}{10 r}$ to produce the last line.
This completes the proof of \eqref{eq:fnBound} by induction.

Next we verify \eqref{eq:fInvnBound} in a similar fashion.
Equation \eqref{eq:fInvnBound} holds with $n = 0$ because in this case both sides are zero.
Suppose \eqref{eq:fInvnBound} holds and $f^{-(n+1)}(v) \in (a,b)$ for some $n \ge 0$.
By \eqref{eq:fnBound} (applied to $u = f^{-(n+1)}(v)$ and with $n+1$ in place of $n$),
\begin{equation}
|v| \le \lambda^{n+1} r + \frac{\lambda^{n+1}}{5 r} \,r^2 = \frac{6 \lambda^{n+1} r}{5}.
\label{eq:fnProof-29}
\end{equation}
The induction hypothesis gives
\begin{equation}
\left| f^{-n}(v) \right|
\le \frac{|v|}{\lambda^n} + \frac{2 v^2}{5 \lambda^{2 n} r}
\le \left( 1 + \frac{12 \lambda}{25} \right) \frac{|v|}{\lambda^n}
\le \frac{2 |v|}{\lambda^n},
\label{eq:fnProof-30}
\end{equation}
where we have used \eqref{eq:fnProof-29} for the middle inequality.
Then by \eqref{eq:fnProof-fInvBound}, \eqref{eq:fnProof-30}, and the induction hypothesis,
\begin{align}
\left| f^{-(n+1)}(v) - \frac{v}{\lambda^{n+1}} \right|
&\le \left| f^{-1} \left( f^{-n}(v) \right) - \frac{f^{-n}(v)}{\lambda} \right|
+ \frac{1}{\lambda} \left| f^{-n}(v) - \frac{v}{\lambda^n} \right| \nonumber \\
&\le \frac{K}{\lambda^3} \left( \frac{2 |v|}{\lambda^n} \right)^2 + \frac{1}{\lambda} \,\frac{2 \left( 1 - \lambda^n \right)}{5 \lambda^{2 n} r} \,v^2 \nonumber \\
&= \frac{2 \left( 1 - \lambda^{n+1} \right)}{5 \lambda^{2(n+1)} r} \,v^2, \nonumber
\end{align}
where we have substituted $K = \frac{\lambda(1-\lambda)}{10 r}$ to produce the last line.
This completes the proof of \eqref{eq:fInvnBound} by induction.
\end{proof}

\begin{proof}[Proof of Lemma \ref{le:bound}]
Without loss of generality suppose $x^* = 0$ and $y^* = 0$.
By Lemma \ref{le:fn}a applied to the map $f$ and the point $u = b$,
\begin{equation}
\left| f^n(b) - \lambda^n b \right| \le \frac{\lambda^n b^2}{5 r}.
\label{eq:boundProof-10}
\end{equation}
This implies
\begin{equation}
\left| f^n(b) \right| \ge \lambda^n b - \frac{\lambda^n b^2}{5 r} \ge \frac{4 \lambda^n b}{5},
\label{eq:boundProof-11}
\end{equation}
using the assumption $b \le r$.
Similarly by Lemma \ref{le:fn}a applied to the map $g$ and the point $u = d$,
\begin{equation}
\left| g^n(d) - \lambda^n d \right| \le \frac{\lambda^n d^2}{5 r}.
\label{eq:boundProof-20}
\end{equation}
We now show the value of $\frac{g^n(d)}{f^n(b)}$ is close to $\chi$,
where $\chi = \frac{d}{b}$ because $x^* = 0$ and $y^* = 0$.
Specifically, by \eqref{eq:boundProof-10}, \eqref{eq:boundProof-11}, and \eqref{eq:boundProof-20},
\begin{align}
\left| \frac{g^n(d)}{f^n(b)} - \chi \right|
&= \frac{1}{f^n(b)} \left| g^n(d) - \tfrac{d}{b} f^n(b) \right| \nonumber \\
&\le \frac{1}{f^n(b)} \left( \left| g^n(d) - \lambda^n d \right|
+ \tfrac{d}{b} \left| f^n(b) - \lambda^n b \right| \right) \nonumber \\
&\le \frac{5}{4 \lambda^n b} \left( \frac{\lambda^n d^2}{5 r} + \frac{\lambda^n b d}{5 r} \right) \nonumber \\
&= \frac{\chi (b + d)}{4 r}. \nonumber
\end{align}
But $\chi \le \frac{3}{2}$, so
\begin{equation}
\left| \frac{g^n(d)}{f^n(b)} - \chi \right| \le \frac{\frac{3}{2} \left( 1 + \frac{3}{2} \right) b}{4 r} = \frac{15 b}{16 r} < \frac{b}{r}.
\label{eq:boundProof-30}
\end{equation}

Now fix $x \in (a,b) \setminus \{ 0 \}$.
Since $f^n(x) \to 0$ as $n \to \infty$,
we have $\frac{h(f^n(x))}{f^n(x)} \to h'(0)$ as $n \to \infty$.
But $h \circ f^n = g^n \circ h$,
so $\frac{g^n(h(x))}{f^n(x)} \to h'(0)$ as $n \to \infty$.
Also with $b$ and $d$ in place of $x$ and $h(x)$,
we have $\frac{g^n(d)}{f^n(b)} \to h'(0)$ as $n \to \infty$.
So there exists $n \in \mathbb{Z}$ such that
\begin{equation}
\left| \frac{g^n(h(x))}{f^n(x)} - \frac{g^n(d)}{f^n(b)} \right| \le \frac{b}{2 r}.
\nonumber
\end{equation}
Combining this with \eqref{eq:boundProof-30} gives
\begin{equation}
\left| \frac{g^n(h(x))}{f^n(x)} - \chi \right| \le \frac{3 b}{2 r}.
\label{eq:boundProof-41}
\end{equation}
Also $\chi < \frac{3}{2}$ so
\begin{equation}
\left| \frac{g^n(h(x))}{f^n(x)} \right| < \frac{3}{2} + \frac{3 b}{2 r} \le \frac{3}{2} + \frac{3}{2} = 3.
\label{eq:boundProof-42}
\end{equation}
By Lemma \ref{le:fn}a applied to the map $f$ and the point $u = x$,
\begin{equation}
\left| f^n(x) - \lambda^n x \right| \le \frac{\lambda^n x^2}{5 r}.
\label{eq:boundProof-50}
\end{equation}
Then by \eqref{eq:boundProof-41}, \eqref{eq:boundProof-42}, and \eqref{eq:boundProof-50},
\begin{align}
\left| \frac{g^n(h(x))}{\lambda^n x} - \chi \right|
&= \left| \frac{g^n(h(x))}{\lambda^n x f^n(x)} \left( f^n(x) - \lambda^n x \right) + \frac{g^n(h(x))}{f^n(x)} - \chi \right| \nonumber \\
&\le \frac{g^n(h(x))}{\lambda^n |x| f^n(x)} \left| f^n(x) - \lambda^n x \right|
+ \left| \frac{g^n(h(x))}{f^n(x)} - \chi \right| \nonumber \\
&\le \frac{3}{\lambda^n |x|} \frac{\lambda^n x^2}{5 r} + \frac{3 b}{2 r} \nonumber \\
&= \frac{3 |x|}{5 r} + \frac{3 b}{2 r}.
\label{eq:boundProof-53}
\end{align}
Also \eqref{eq:boundProof-50} implies
\begin{equation}
\left| f^n(x) \right| \le \lambda^n |x| + \frac{\lambda^n x^2}{5 r} \le \frac{6 \lambda^n |x|}{5},
\nonumber
\end{equation}
using $|x| < r$, so by \eqref{eq:boundProof-42},
\begin{equation}
\left| g^n(h(x)) \right| \le \frac{18 \lambda^n |x|}{5}.
\label{eq:boundProof-56}
\end{equation}
By Lemma \ref{le:fn}b applied to the map $g$ and the point $v = g^n(h(x))$,
\begin{equation}
\left| h(x) - \frac{g^n(h(x))}{\lambda^n} \right| \le \frac{2}{5 \lambda^{2n} r} \big( g^n(h(x)) \big)^2
\le \frac{2 (18)^2 x^2}{5 (5)^2 r} < \frac{27 x^2}{5 r},
\nonumber
\end{equation}
where we have used \eqref{eq:boundProof-56} for the middle inequality.
By dividing by $x$ and combining with \eqref{eq:boundProof-53} we obtain
\begin{equation}
\left| \frac{h(x)}{x} - \chi \right| < \frac{6 |x|}{r} + \frac{3 b}{2 r},
\nonumber
\end{equation}
which is the desired bound \eqref{eq:bound}.
\end{proof}


\end{document}